\documentclass[11pt]{amsart}

\usepackage{amsmath}
\pdfoutput=1 
\usepackage{amsfonts}
\usepackage{amssymb}
\usepackage{amsthm}
\usepackage{thmtools}
\usepackage{url}
\usepackage{color}
\allowdisplaybreaks

\usepackage{upgreek}
\usepackage{subcaption}
\usepackage{tikz-cd}
\usepackage{bigdelim}
\usepackage{adjustbox}
\usepackage{multirow}
\usepackage{hhline}
\usepackage{enumerate}
\usepackage{amsthm,verbatim}
\usepackage{graphicx}
\usepackage{enumitem}
\usepackage[capitalize]{cleveref}
\usepackage{comment}
\usepackage{latexsym}
\usepackage{fancyhdr}
\usepackage{import}
\usepackage[margin=1in]{geometry}

\DeclareMathOperator*{\re}{Re}

\DeclareMathOperator*{\diam}{diam}
\DeclareMathOperator*{\interior}{int}
\DeclareMathOperator*{\exterior}{ext}

\DeclareMathOperator*{\length}{length}

\DeclareMathOperator*{\sys}{sys}

\DeclareMathOperator*{\dist}{distor}
\DeclareMathOperator*{\ddist}{ddistor}

\DeclareMathOperator*{\genus}{genus}

\newcommand*{\bound}[1]{\delta(#1)}

\DeclareMathOperator*{\conlength}{convol_1}

\DeclareMathOperator*{\modulo}{mod}
\DeclareMathOperator*{\ess}{ess}

\DeclareMathOperator{\odd}{odd}
\DeclareMathOperator{\even}{even}

\renewcommand{\epsilon}{\varepsilon}
\renewcommand{\Lambda}{\Uplambda}

\newtheorem{theorem}{Theorem}[section]
\newtheorem{lemma}[theorem]{Lemma}

\newtheorem{cor}[theorem]{Corollary}
\newtheorem{prop}[theorem]{Proposition}

\theoremstyle{definition}
\newtheorem{definition}[theorem]{Definition}

\newtheorem{question}[theorem]{Question}
\newtheorem{remark}[theorem]{Remark}

\newtheorem{notation}[theorem]{Notation}
\newtheorem{construction}[theorem]{Construction}

\begin{document}

\title[Knot distortion via curve combinatorics]{Knot distortion via curve combinatorics}
\date{\today}
\author[Vasudevan]{Sahana Vasudevan}
\address[Sahana Vasudevan]{School of Mathematics, Institute for Advanced Study, Princeton, NJ 08540, USA; Department of Mathematics, Princeton University, Princeton, NJ 08540, USA}
\email{sahanav@ias.edu}

\fancyhead[CE]{Vasudevan}

\begin{abstract} 

Distortion is a geometric knot invariant introduced by Gromov in 1983. Since then, the question of estimating distortion for general knots has remained widely open. In this paper, we develop a combinatorial language to study this question.

We apply our method to prove new lower bounds for knot distortion. Specifically, we prove that the distortion of $T_{p, p+1}\# K$ is at least $p$ up to a constant, independent of $K$. In the appendix, we also prove that any embedding of a minimal genus Seifert surface for $T_{p,p+1}\# K$ in $\mathbb{R}^3$ has small extrinsic systole, in the sense that it contains a non-contractible loop with small $\mathbb{R}^3$-diameter relative to the length of the knot. These results are consequences of combinatorial properties of the monodromy map associated to torus knots.
\end{abstract} 

\maketitle

\section{Introduction}

In \cite{Gro83}, Gromov introduced the notion of distortion of a knot. Distortion measures the smallest bi-Lipschitz constant required to embed the knot in $\mathbb{R}^3$. 

\begin{definition}\label{distortion of knot} Let $\beta\subset \mathbb{R}^3$ be a rectifiable closed curve. Then $$\dist(\beta)=\sup_{x,y\in \beta}\frac{d_\beta(x,y)}{|x-y|}.$$ Given a knot $K$, $$\dist(K)=\inf_{\beta} \dist(\beta),$$ where the infimum is over all embeddings $\beta\subset \mathbb{R}^3$ of $K$.
\end{definition}

Since the introduction of knot distortion, a natural question has been:

\begin{question}\label{estimate distortion} How do we estimate $\dist(K)$ for a general knot $K$? In particular, how do we prove lower bounds for $\dist(K)$? 
\end{question}

(Note that upper bounds for distortion may be proved by simply constructing an example.) \cref{estimate distortion} lies at the intersection of the fields of knot theory and quantitative topology. However, it has been very difficult to answer, even for specific knots. In this paper, we introduce a systematic method to study it. The core of the method is to develop tools to study the following related, but purely topological question:

\begin{question}\label{study intersection of ball and knot} How do we understand intersections of an embedded ball in $\mathbb{R}^3$ with a knot $K$ (up to homeomorphisms of $\mathbb{R}^3$ preserving the ball and knot)?
\end{question}

This question is a little imprecise compared to \cref{estimate distortion}, but we now explain how they are related at a high level. In the field of quantitative topology, one proves quantitative versions of qualitative statements in topology. Proofs of these quantitative versions often operate by finding a proof of the underlying topological statement that is amenable to quantification, and keeping track of the quantitative notions in order to obtain the quantitative statement. (For example, the isoperimetric inequality due to Federer and Fleming in \cite{FF60}, may be viewed as a quantification of the statement that Euclidean spaces have trivial homology.) At a high level (and after several modifications), \cref{estimate distortion} can be viewed as a quantification of \cref{study intersection of ball and knot}.

\cref{study intersection of ball and knot} does not appear to be a standard question in knot theory. Indeed, standard knot invariants cannot be quantified to give bounds for \cref{estimate distortion}, as they are usually unbounded on the set of knots with uniformly bounded distortion (see \cite[Section 4]{GG12}). In this paper, we introduce a combinatorial language to study \cref{study intersection of ball and knot}, for fibered knots. We do this by developing a dictionary between intersections of embedded balls in $\mathbb{R}^3$ with the knot, and certain combinatorial objects arising from curves on the fiber surface. The dictionary allows us to rigorously formulate and prove statements about \cref{study intersection of ball and knot}. Roughly speaking, quantifying these proofs give proofs of estimates about knot distortion. 

Our method works not just for fibered knots, but also for knots with a fibered component. In this paper, we use this method to prove new lower bounds for knot distortion. To explain the context behind these new bounds, let us first describe previous work on \cref{estimate distortion}. Even simpler versions of this question have been difficult to answer. In \cite{Gro83}, Gromov asked if there are knots with arbitrarily large distortion. This question was open until \cite{Par11}, when Pardon proved that torus knots $T_{p,q}$ have distortion at least $\min\{p,q\}$, up to a constant. Subsequently, Guth and Gromov constructed other examples of knots with arbitrarily large distortion in \cite{GG12}; till now, these have been the only examples of nontrivial lower bounds for knot distortion. As an application of our method, we prove uniform lower bounds for a new class of knots: the connected sum $T_{p,p+1}\# K$ for arbitrary $K$, answering a question of Pardon.

\begin{theorem}\label{distortion of connected sum is large} For any knot $K$, $$\dist(T_{p,p+1}\#K)\gtrsim p.$$ 
\end{theorem}

All knots in this paper are assumed to be tame. The symbol $\gtrsim$ means the inequality is true up to a constant multiplicative factor; the constant is universal. (Henceforth, the symbols $\gtrsim$ and $\lesssim$ will always be used in this manner.) The inequality is sharp, as a standard embedding of $T_{p,p+1}$ has distortion around $p$. 

Besides the results in \cite{Par11}, \cite{GG12}, and this paper, there have been very few results about knot distortion. Gromov proved that the distortion of a closed curve is at least $\pi/2$, with equality if and only if the curve is a circle. In \cite{DS09}, Denne and Sullivan proved that $\dist(K)\geq 5\pi/3$ for all nontrivial knots $K$. The sharpness of Pardon's inequality $\dist(T_{p,q})\gtrsim \min\{p,q\}$ is also not known in general. Although the inequality is sharp when $p$ and $q$ are both large, it is for example open whether $\dist(T_{2,q})\to \infty$ as $q\to \infty$. In \cite{Stu16}, Studer proved that the distortion of $T_{2,q}$ grows sublinearly: $\dist(T_{2,q})\lesssim q/\log q$. 

\subsection{Topological questions underlying knot distortion}\label{topology underlying distortion} We now explain the relationship between \cref{estimate distortion} and \cref{study intersection of ball and knot} in more detail.

Distortion of a curve splits into two other distortion-like invariants: conformal length and diametrical distortion. The concept of conformal length was first used in the proof of the main theorem in \cite{Par11}. The definition appeared in \cite{GG12}; it is a slight variation of the conformal volume defined in \cite{LY82}.

\begin{definition} Let $\beta\subset \mathbb{R}^3$ be a rectifiable closed curve. The conformal length of $\beta$ is $$\conlength(\beta)=\sup_{r>0} \sup_{x\in \mathbb{R}^3} \frac{\length(B(x,r)\cap \beta)}{r}.$$
\end{definition}

It is an elementary fact that conformal length is at most distortion, up to a constant. Diametrical distortion measures the difference: 

\begin{definition} Let $\beta\subset \mathbb{R}^3$ be a rectifiable closed curve. Let $\diam_{\beta}(x,y)$ be the minimum of the diameters of the two arcs of $\beta$ connecting $x$ and $y$. Then the diametrical distortion of $\beta$ is $$\ddist(\beta)=\sup_{x,y\in \beta}\frac{\diam_{\beta}(x,y)}{|x-y|}.$$ 
\end{definition}

To bound distortion of a curve on either side, it suffices to bound conformal length and diametrical distortion (see \cref{distortion comparison} for a precise statement). So understanding distortion (up to some error terms) reduces to understanding the latter two invariants. 

To make sense of the topology underlying conformal length, it is useful to think of conformal length at various scales, i.e. $$\conlength(\beta,r)=\sup_{x\in \mathbb{R}^3} \frac{\length(B(x,r)\cap \beta)}{r}.$$ If $\conlength(\beta,r)$ is small, then for most $s\in [r/2,r]$, $|\partial B(x,r)\cap \beta|$ is small. The latter statement is a quantitative way of understanding how Euclidean spheres intersect the curve $\beta$. (The statement that $\conlength(\beta,r)$ is large also has a similar description, but it would involve averaging over a family of Euclidean balls with different centers and we shall not write it out explicitly.)

Diametrical distortion also has such a description. We may think of a ball intersecting the curve $\beta$ as containing a certain collection of strands of $\beta$. Now, $\ddist(\beta)\lesssim \alpha$ if and only if for all $r>0$ and $x\in \mathbb{R}^3$, if $\beta_1$ and $\beta_2$ are two distinct strands of $\beta$ in $B(x,r)\cap \beta$, then they lie in the same strand of $\beta$ in $B(x,\alpha r)\cap \beta$.  

So both conformal length and diametrical distortion of a curve have descriptions in terms of how Euclidean spheres intersect the curve. Now, like distortion, these invariants can be defined for a knot too. The conformal length and diametrical distortion of a knot have descriptions in terms of how the knot intersects topological spheres. 

Let us take conformal length as an example to illustrate this. Let $K$ be a knot and suppose the conformal length of $K$ is small. Then the conformal length of some curve $\beta$ representing $K$ is small. This means that: 
\begin{multline}\text{for all }x\in \mathbb{R}^3,\text{ for all }r_0>0, \text{ for most balls of radius } r\sim r_0 \text{ around }x, \\ \text{ }|\partial B(x,r)\cap \beta| \text{ is small}. \tag{$\ast$}
\end{multline} 

In $(\ast)$, there are two key objects: the curve $\beta$, and the foliation of $\mathbb{R}^3$ by Euclidean spheres $\partial B(x,r)$ around $x$. Now, fix a representative $\beta_0$ of $K$. (We may assume that $\beta$ and $\beta_0$ both avoid $x$.) Consider a homeomorphism of $\mathbb{R}^3-\{x\}$ taking $\beta$ to $\beta_0$. Under this homeomorphism, the foliation $\{\partial B(x,r)\}$ is sent to another foliation $\{S_r\}$ of $\mathbb{R}^3-\{x\}$ by topological spheres around $x$. Moreover, any foliation of $\mathbb{R}^3-\{x\}$ by topological spheres around $x$ can arise in this way. Under the homeomorphism, statement $(*)$ translates to: 

\begin{equation} \text{ for all }r_0>0,\text{ for most }r\sim r_0,\text{ } |S_r\cap \beta_0|\text{ is small}. \tag{$\ast \ast$}
\end{equation}

The statement $(**)$ is a quantitative statement about how topological balls intersect the knot. (The diametrical distortion of a knot also has a similar description. However, for the proof of \cref{distortion of connected sum is large}, the conformal length will be sufficient.) Note that $(**)$ is a statement about a family of topological balls (foliation by topological spheres), not just one ball. Nevertheless, \cref{study intersection of ball and knot} is the most  fundamental topological question in this direction, and results on it already have consequences for knot distortion.

\subsection{The dictionary} We next give an overview of the dictionary that enables us to study \cref{study intersection of ball and knot}. Roughly speaking, the dictionary gives us a combinatorial way to understand the intersection of an embedded $3$-submanifold (for example, embedded ball) and a knot in $\mathbb{R}^3$. 

Rigorously, let $K$ be a fibered knot, and denote by $\Sigma$ a fiber Seifert surface. An embedded $3$-submanifold in $\mathbb{R}^3$ gives an embedded $3$-submanifold in the knot complement $S^3-N(K)$, constructed by removing a neighborhood of the knot from the ball. Since $K$ is fibered, the $3$-submanifold of the knot complement also gives a $3$-submanifold in the produce $\Sigma\times[0,1]$, constructed by cutting the knot complement along a fiber Seifert surface.  

Understanding the structure of a $3$-submanifold of $\Sigma\times [0,1]$ amounts to understanding its interior boundary, i.e. the part of its boundary lying in the interior of the product. The interior boundary is a properly embedded surface of the product. By doing a sequence of operations like compressions and $\partial$-compressions, the interior boundary may be put into a standard form (i.e. product). 

Let $L$ be a $3$-submanifold of $\Sigma\times [0,1]$, with interior boundary $S$. Then $L\cap (\Sigma\times \{0\})$ is a subsurface of $\Sigma$. A $\partial$-compression applied to $S$, along a disk which intersects $\partial (\Sigma\times [0,1])$ only at $\Sigma\cap \{0\}$, gives a new properly embedded surface which is the interior boundary of a new $3$-submanifold. This $3$-submanifold also intersects $\Sigma\cap \{0\}$ in a subsurface. The old subsurface (which is the intersection of $\Sigma\cap \{0\}$ with $L$) and the new subsurface are related by a certain type of combinatorial operation on subsurfaces, which we call an elementary move (see Figure 2).

\begin{figure}[h]
%% Creator: Inkscape 1.1.2 (0a00cf5339, 2022-02-04), www.inkscape.org
%% Accompanies image file 'Fig1.pdf' (pdf, eps, ps)
\begingroup%
  \makeatletter%
  \providecommand\color[2][]{%
    \errmessage{(Inkscape) Color is used for the text in Inkscape, but the package 'color.sty' is not loaded}%
    \renewcommand\color[2][]{}%
  }%
  \providecommand\transparent[1]{%
    \errmessage{(Inkscape) Transparency is used (non-zero) for the text in Inkscape, but the package 'transparent.sty' is not loaded}%
    \renewcommand\transparent[1]{}%
  }%
  \providecommand\rotatebox[2]{#2}%
  \newcommand*\fsize{\dimexpr\f@size pt\relax}%
  \newcommand*\lineheight[1]{\fontsize{\fsize}{#1\fsize}\selectfont}%
  \ifx\svgwidth\undefined%
    \setlength{\unitlength}{109.31963292bp}%
    \ifx\svgscale\undefined%
      \relax%
    \else%
      \setlength{\unitlength}{\unitlength * \real{\svgscale}}%
    \fi%
  \else%
    \setlength{\unitlength}{\svgwidth}%
  \fi%
  \global\let\svgwidth\undefined%
  \global\let\svgscale\undefined%
  \makeatother%
  \begin{picture}(1,1.12969348)%
    \lineheight{1}%
    \setlength\tabcolsep{0pt}%
    \put(0,0){\includegraphics[width=\unitlength,page=1]{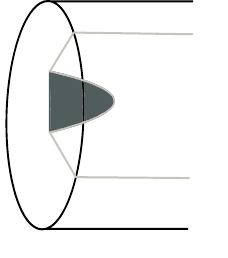}}%
    \put(-0.0021954,0.0207099){\makebox(0,0)[lt]{\lineheight{1.25}\smash{\begin{tabular}[t]{l}$\Sigma_{p,p+1}\times \{0\}$\end{tabular}}}}%
    \put(0.87138346,0.94951715){\makebox(0,0)[lt]{\lineheight{1.25}\smash{\begin{tabular}[t]{l}$S$\end{tabular}}}}%
  \end{picture}%
\endgroup%

\caption{A $\partial$-compressing disk for $S$.}
\end{figure}

\begin{figure}[h]
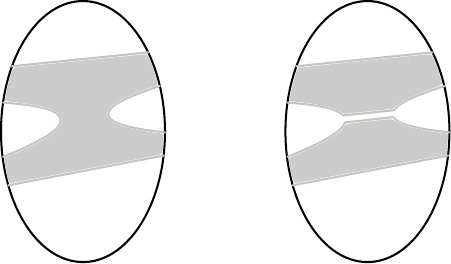
\caption{An elementary move applied to a subsurface.}
\end{figure}

Using this observation, we may obtain a description of $L$ in terms of subsurfaces of $\Sigma$. Specifically, a $L$ gives a sequence of subsurfaces of $\Sigma$, in which each subsurface is obtained by applying an elementary move to the previous one. We call such a sequence a path of subsurfaces. In the opposite direction, a path of subsurfaces also produces a $3$-submanifold of the product. This $3$-submanifold comes from a $3$-submanifold of the knot complement if it can be glued back together. In the language of subsurface paths, this means that the first and last terms of the path must be subsurfaces related by the monodromy map. So the original $3$-submanifold in $\mathbb{R}^3$ has such a combinatorial description which captures how it intersects with not only the knot, but also the Seifert surface.

One of the limitations of this combinatorial description is that it is not clear when two different descriptions correspond to the same configuration of $3$-submanifold intersecting knot. (Here, ``same'' means that the there is a homeomorphism of $\mathbb{R}^3$ preserving the $3$-submanifold and the knot.) Nevertheless, it gives us a rigorous language to formulate and prove statements about \cref{study intersection of ball and knot}. In the next section, we will explain how this helps us prove bounds about knot distortion.

\subsection{Ideas in the proof of \cref{distortion of connected sum is large}} We now illustrate how the dictionary is used, by examining our proof of \cref{distortion of connected sum is large} in the specific case where $K$ is the unknot, i.e. the inequality $\dist(T_{p,p+1})\gtrsim p$. Although the latter inequality was already proved in \cite{Par11}, we will see through this example how the dictionary comes into play in order to construct an argument that generalizes to the the case of $T_{p,p+1}\# K$. 

In \cite{Par11}, a key topological input is the following observation:

\begin{lemma}\label{Pardon lemma} Let $\beta$ be a $T_{p,p+1}$ torus knot lying on the surface of a torus $T$. If $S$ is a sphere intersecting $T$ in an essential curve, then $|S\cap \beta|\gtrsim p$. 
\end{lemma}

This lemma states that if the interior and exterior of $S$ both intersect $T$ in a nontrivial way, then $S$ intersects the knot many times. In this paper, we prove a version of the lemma that applies to a Seifert surface associated to the knot. 

\begin{theorem}\label{topological lemma} Let $\beta\subset \mathbb{R}^3$ be an embedding of $T_{p,p+1}$. Let $\Sigma_{p,p+1}\subset \mathbb{R}^3$ be a standard Seifert surface, with $\partial \Sigma_{p,p+1}=\beta$. Let $S$ be an embedded $S^2$ in $\mathbb{R}^3$ intersecting $\Sigma_{p,p+1}$ transversely. Denote by $\interior(S)$ the closed ball that $S$ bounds, and $\exterior(S)$ its complement in $\mathbb{R}^3$. Suppose $$\frac{p(p-1)}{20}\leq \genus(\interior(S)\cap \Sigma_{p,p+1}), \genus(\exterior(S)\cap \Sigma_{p,p+1})\leq \frac{9p(p-1)}{20}.$$ Then $$|S\cap \beta|\gtrsim p,$$ with constant independent of $p$ and $S$.
\end{theorem}

This theorem actually generalizes to the case of $T_{p,p+1}\#K$. Roughly speaking, the generalization also states that if both the interior and exterior of the ball intersect the Seifert surface in a suitably large way, then the boundary of the ball intersects the knot many times; however, in this case the notion of ``suitably large'' is more difficult to precisely formulate. So the generalization is stated only in \cref{Topological lemma section} as \cref{topological lemma for connected sum}.

\cref{topological lemma} and \cref{topological lemma for connected sum} lie within the domain of the dictionary. For example, \cref{topological lemma} is a statement about embedded balls which has a hypothesis and and a conclusion. The hypothesis is that the ball intersects the Seifert surface in a certain way. The conclusion is that the boundary of the ball intersects the knot many times. Using the dictionary, both the hypothesis and the conclusion have combinatorial characterizations. In order to prove the theorem, we need to verify that the hypothesis implies the conclusion on the combinatorial side. In practice, the translation to the combinatorial side is not as direct; while the combinatorial side gives a language to describe the objects of study in \cref{study intersection of ball and knot}, it also obscures the geometric intuition. So our actual proof of \cref{topological lemma for connected sum} uses arguments on both the geometric side and the combinatorial side. 

We now turn to the distortion bounds; the bound for $\dist(T_{p,p+1})$ follows from \cref{Pardon lemma} along with a double bubble argument in \cite{Par11}. Similarly,  \cref{distortion of connected sum is large} follows from the generalization of \cref{topological lemma}, along with an adaptation of the double bubble argument. We expect an analogue of \cref{topological lemma} to be true for $T_{p,q}$, for all $p$ and $q$. However, in order to prove a similar theorem for $T_{p,q}\# K$, we need the appropriate generalization of \cref{topological lemma for connected sum} for all $p$ and $q$. This generalization depends on number theoretic properties of $p$ and $q$, and gives worse bounds for $p$ and $q$ for which $p^{-1}(\modulo q)$ and $q^{-1} (\modulo p)$ are large. So our method would give an analogue of \cref{distortion of connected sum is large} for general $p$ and $q$, but they are not necessarily sharp, even when $p$ and $q$ are both large. For example, instead of an inequality $\dist(T_{p,q}\#K)\gtrsim \min\{p,q\}$, we would get an extra multiplicative factor that depends on $p^{-1}(\modulo q)$ and $q^{-1} (\modulo p)$. For readability, we restrict our statements and proofs to the case of $p$ and $p+1$.

\subsection{Structure of the paper} In \cref{Preliminaries}, we record some basic results about torus knots. 

In \cref{Subsurfaces and combinatorics}, \cref{paths of subsurfaces}, and \cref{Surfaces in a product}, we develop the prerequisite theory to formulate the dictionary. Specifically, in \cref{Subsurfaces and combinatorics}, we develop a theory of multicurves and subsurfaces on a surface with marked points on its boundary. We also introduce certain combinatorial operations and invariants associated to subsurfaces. In \cref{paths of subsurfaces}, we develop a theory of paths of subsurfaces. In \cref{Surfaces in a product}, we introduce some definitions of $\partial$-incompressibility that apply to surfaces in $3$-manifolds suited to our context. We classify surfaces in a product $3$-manifold that are incompressible and $\partial$-incompressible according to our definitions. 

In \cref{Dictionary}, we formulate the dictionary for fibered knots. At the end of this section, we prove some specific results about the dictionary that apply to torus knots. 

In \cref{Topological lemma section}, we prove the key topological statement that implies the main theorems: we use the combinatorial invariants introduced in \cref{Subsurfaces and combinatorics} to formulate a generalization of \cref{topological lemma} that applies to $T_{p,p+1}\# K$, and  prove it using the dictionary. Finally, we prove \cref{distortion of connected sum is large} in \cref{Proofs of the main theorems}. 

\subsection*{Acknowledgments} I thank Larry Guth, Helmut Hofer and Shmuel Weinberger for many conversations related to this paper. I thank Peter Ozsvath and Akshay Venkatesh for answering several of my questions. I have been supported by NSF DMS-2202831 and the Friends of the Institute for Advanced Study.

\section{Torus knots} \label{Preliminaries}

In this section, we note some basic results about torus knots. Let $T\subset \mathbb{R}^3$ be an unknotted embedded torus. Note that $\mathbb{R}^3-T$ has two connected components. Let $u\in H_1(T)$ be the primitive homology class that is trivial in the unbounded connected component of $\mathbb{R}^3-T$. Let $v\in H_1(T)$ be the primitive homology class that is trivial in the bounded connected component of $\mathbb{R}^3-T$. For relatively prime integers $p$ and $q$, the torus knot $T_{p,q}$ is isotopic to a simple closed curve lying on $T$, representing the homology class $pu+qv$. We now give an analytic description of $T_{p,q}$, which will also allow us to explicitly describe the fiber bundle associated to $T_{p,q}$.

\subsection{Torus knots via Milnor fibrations} \label{torus knot via milnor fibration}

In this section, we give a description of $T_{p,q}$ as the singularity set of a polynomial. This point of view is due to Milnor, in \cite{Mil68}. 

Consider the hypersurface $$G_{p,q}=\{(x,y)\in \mathbb{C}^2|x^p-y^{q}=0\}$$ in $\mathbb{C}^2$. Let $S^3$ be the unit sphere around the origin in $\mathbb{C}^2$. The torus knot $T_{p,q}$ is the intersection $$G_{p,q}\cap S^3\subset S^3.$$

Let $a$ and $b$ be positive real numbers such that $a^2+b^2=1$ and $a^p=b^{q}$. Then $T_{p,q}$ lies on the torus $$T=\{(x,y)\in \mathbb{C}^2||x|=a,|y|=b\}.$$ There is a fiber bundle \begin{gather*}S^3-T_{p,q}\to S^1\\
(x,y)\to \frac{x^p-y^{q}}{|x^p-y^{q}|}.
\end{gather*}
A fiber $$\Sigma_{p,q}=\{(x,y)\in S^3|x^p-y^{q}\in \mathbb{R}^+\}$$ is a standard Seifert surface for $T_{p,q}$. The monodromy map associated to the fiber bundle is 
\begin{gather*}\Sigma_{p,q}\to \Sigma_{p,q}\\ (x,y)\to (e^{2\pi i/p}x,e^{2\pi i/q}y).
\end{gather*}

\subsection{Topological description of monodromy map}\label{topological description of monodromy} 

In this section, we use the analytic description of the monodromy map from \cref{torus knot via milnor fibration} to give an explicit topological description of the map in terms of a CW structure on the Seifert surface.

\begin{lemma}\label{monodromy of torus knot} The Seifert surface surface $\Sigma_{p,q}$ admits the following CW structure so that the monodromy is a cellular map. The $2$-cells are two sets $$U_1,...,U_p$$ and $$V_1,...,V_{q}$$ of $p$ and $q$ disks each. Each $U_i$ is glued to each $V_j$ along a $1$-cell $\alpha_{i,j}$. The monodromy map $$f:\Sigma_{p,q}\to \Sigma_{p,q}$$ is a cellular map that satisfies \begin{equation*}
\begin{cases} f(U_i)=U_{i+1 (\modulo p)}\\
f(V_j)=V_{j+1(\modulo q)}\\
f(\alpha_{i,j})=\alpha_{i+1(\modulo p),j+1 (\modulo q)}.
\end{cases}
\end{equation*}
\end{lemma}

\begin{proof} We construct the above CW structure on $\Sigma_{p,q}$ using the analytic description of $\Sigma_{p,q}$. First, we construct the $1$-cells $\alpha_{i,j}$. To do this, consider the intersection $$\Sigma_{p,q}\cap T=\{(x,y)\in \mathbb{C}^2||x|^p=|y|^{q}=a^p=b^{q}, x^p-y^{q}\in \mathbb{R}^+\}.$$ Consider $x^p$ and $y^{q}$ as points on the $a^p=b^{q}$-radius circle around the origin in $\mathbb{C}$. For any $x^p$ on the circle with $\re x^p\geq 0$, there is exactly one $y^{q}$ on the circle such that $x^p-y^{q}\in \mathbb{R}^+$. (For $x^p$ on the circle with $\re x^p<0$, there is no $y^{q}$ on the circle with $x^p-y^{q}\in \mathbb{R}^+$.) So the locus $$\{(x^p,y^{q})\in \mathbb{C}^2||x|^p=|y|^{q}=a^p=b^{q}, x^p-y^{q}\in \mathbb{R}^+\}$$ is a connected arc in $\mathbb{C}^2$ (with endpoints where $x^p-y^{q}=0$). Thus, $\Sigma_{p,q}\cap T$ is the union of $pq$ arcs which we define to be $\alpha_{i,j}$ ($1\leq i\leq p$ and $1\leq j\leq q$). Each $\alpha_{i,j}$ has endpoints on $T_{p,q}$. The monodromy map sends $\alpha_{i,j}$ to $\alpha_{i+1(\modulo p),j+1 (\modulo q)}$ as desired.

To construct the $2$-cells of the CW structure, we consider the intersections of the interior and exterior of $T$ with $\Sigma_{p,q}$.

Let $$\interior(T)=\{(x,y)\in S^3||x|<a, |y|>b\}.$$ There is a continuous map 
\begin{gather*}\interior(T)\cap \Sigma_{p,q}\to \{x\in \mathbb{C}||x|<|a|\}\\ (x,y)\to x.
\end{gather*} For any $x^p\leq a^p$, there exists a unique $y^{q}$ such that $|x|^2+|y|^2=1$ and $x^p-y^{q}\in \mathbb{R}^+$. Note that $y^{q}$ is never $0$. So our continuous map is actually a $q$-sheeted cover of a disk. Therefore $\interior(T)\cap \Sigma_{p,q}$ is the disjoint union of $q$ disks that we label $V_1,...,V_{q}$. The monodromy map sends $V_j$ to $V_{j+1(\modulo q)}$ 

Similarly, $\exterior(T)\cap \Sigma_{p}$ is the disjoint union of $p$ disks that we label $U_1,...,U_{p}$, and the monodromy map sends $U_i$ to $U_{i+1(\modulo p)}$. By construction, each $U_i$ and $V_j$ share $\alpha_{i,j}$ as a boundary.
\end{proof} 

The lemma implies that the knot complement is homeomorphic to the mapping torus: $$S^3-T_{p,q}=\Sigma_{p,q}\times [0,1]/(x,0)\sim (f(x),1).$$

\begin{lemma}\label{torus in fiber bundle} Under this homeomorphism, $T-T_{p,q}$ is  identified with $$\cup_{i,j}\alpha_{i,j} \times [0,1].$$
\end{lemma} 

\begin{proof} The map \begin{gather*}\Sigma_{p,q}\times [0,1]/(x,0)\sim (f(x),1)\to S^3-T_{p,q}\\
(x,y,t)\to (e^{2\pi t i/p}x,e^{2\pi t i/q}y)
\end{gather*} 
is a homeomorphism. Since $T=\{x=|a|,y=|b|\}$, for any $t\in [0,1]$, the map $$(x,y)\to (e^{2\pi t i/p}x,e^{2\pi t i/q}y)$$ preserves $T$. So $$T_{p,q}\cup_{i,j}\alpha_{i,j} \times [0,1]\subset T.$$ Since both are connected closed surfaces, they are identical.
\end{proof}

\section{Subsurfaces and combinatorics} \label{Subsurfaces and combinatorics}

\subsection{Multicurves} Let $\Lambda$ be a compact surface, possibly with boundary. Let $P\subset \partial \Lambda$ be a (possibly empty) set of marked points on the boundary of $\Lambda$. In this section, we define multicurves on a surface in a more general setting than a compact surface without boundary; we define multicurves on the pair $(\Lambda, P)$. When $P=\emptyset$, we denote the pair by just $\Lambda$. See \cite{FM12} for an exposition of curves on closed surfaces and surfaces with boundary.

\begin{definition}\label{multicurves} A multicurve on $(\Lambda,P)$ is a union of closed curves on $\Lambda$ and arcs with boundary on $\partial \Lambda -P$. A multicurve is simple if it has no self-intersections. 
\end{definition} 

We will assume that multicurves on $\Lambda$ intersect $\partial \Lambda$ transversely. A multicurve can be oriented, meaning that each component has an orientation. 

\begin{definition}\label{isotopy of multicurves} Two (oriented) simple multicurves on $(\Lambda, P)$ are isotopic if there is an isotopy of $\Lambda$, identity on $P$, taking one multicurve to the other.
\end{definition}

\begin{definition} \label{essential multicurves} A connected simple multicurve $\gamma$ on $(\Lambda,P)$ is inessential if it is isotopic to a curve in an arbitrarily small neighborhood of a point of $\Lambda$, or isotopic to a closed curve component of $\partial \Lambda-P$. See Figure 3 for a list of the possible types of inessential components. A connected simple multicurve is essential if it is not inessential. A simple multicurve is essential if all of its components are essential. 
\end{definition}

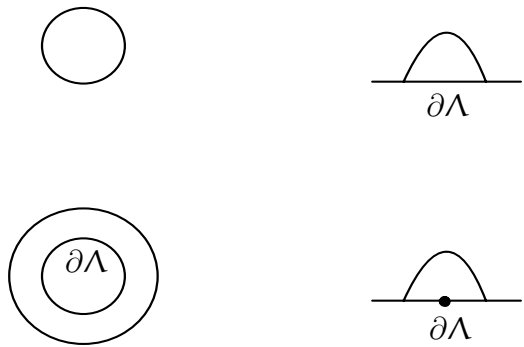
\begin{figure}[h]
%% Creator: Inkscape 1.1.2 (0a00cf5339, 2022-02-04), www.inkscape.org
%% Accompanies image file 'Fig3.pdf' (pdf, eps, ps)
\begingroup%
  \makeatletter%
  \providecommand\color[2][]{%
    \errmessage{(Inkscape) Color is used for the text in Inkscape, but the package 'color.sty' is not loaded}%
    \renewcommand\color[2][]{}%
  }%
  \providecommand\transparent[1]{%
    \errmessage{(Inkscape) Transparency is used (non-zero) for the text in Inkscape, but the package 'transparent.sty' is not loaded}%
    \renewcommand\transparent[1]{}%
  }%
  \providecommand\rotatebox[2]{#2}%
  \newcommand*\fsize{\dimexpr\f@size pt\relax}%
  \newcommand*\lineheight[1]{\fontsize{\fsize}{#1\fsize}\selectfont}%
  \ifx\svgwidth\undefined%
    \setlength{\unitlength}{193.6276332bp}%
    \ifx\svgscale\undefined%
      \relax%
    \else%
      \setlength{\unitlength}{\unitlength * \real{\svgscale}}%
    \fi%
  \else%
    \setlength{\unitlength}{\svgwidth}%
  \fi%
  \global\let\svgwidth\undefined%
  \global\let\svgscale\undefined%
  \makeatother%
  \begin{picture}(1,0.66332286)%
    \lineheight{1}%
    \setlength\tabcolsep{0pt}%
    \put(0,0){\includegraphics[width=\unitlength,page=1]{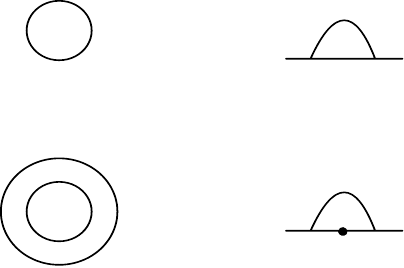}}%
    \put(0.81446904,0.45122004){\makebox(0,0)[lt]{\lineheight{1.25}\smash{\begin{tabular}[t]{l}$\partial \Lambda$\end{tabular}}}}%
    \put(0.11077884,0.14764874){\makebox(0,0)[lt]{\lineheight{1.25}\smash{\begin{tabular}[t]{l}$\partial \Lambda$\end{tabular}}}}%
    \put(0.81960442,0.01741487){\makebox(0,0)[lt]{\lineheight{1.25}\smash{\begin{tabular}[t]{l}$\partial \Lambda$\end{tabular}}}}%
  \end{picture}%
\endgroup%

\caption{Types of inessential components (clockwise from top-left): closed curve bounding a disk, arc bounding a disk with boundary component not containing any points of $P$, arc around a point of $P$, and closed curve isotopic to a boundary component of $\partial \Lambda$ not containing any points of $P$.}
\end{figure}

\begin{definition}\label{essential part of multicurve} Let $\gamma$ be an oriented multicurve on $(\Lambda,P)$. The essential part of $\gamma$, denoted $\gamma^{\ess}$, is the oriented multicurve formed by deleting the inessential components of $\gamma$.
\end{definition}

\begin{definition}\label{bigon} Let $\alpha$ and $\beta$ be essential simple multicurves on $(\Lambda,P)$. A bigon is an embedded disk on $\Lambda$ whose boundary is the union of an arc of $\alpha$ and an arc of $\beta$. A half-bigon is an embedded disk on $\Lambda$ whose boundary is the union of an arc of $\alpha$, an arc of $\beta$, and an arc of $\partial \Lambda-P$.
\end{definition}

\begin{definition}\label{minimal position} Let $\alpha$ and $\beta$ be essential simple multicurves on $(\Lambda,P)$. Then $\alpha$ and $\beta$ are in minimal position if $\alpha$ and $\beta$ do not bound any bigons or half-bigons.
\end{definition}

\begin{prop}[Bigon criterion]\label{bigon criterion} Let $\alpha$ and $\beta$ be essential simple multicurves on $(\Lambda,P)$ in minimal position. Then $\alpha$ and $\beta$ minimize their geometric intersection number over all pairs of multicurves in their respective isotopy classes. 
\end{prop}

A proof (for multicurves on a closed surface or surface with boundary) may be found in \cite[Proposition 1.7, Section 1.2.7]{FM12}; the proof of \cref{bigon criterion} is analogous.

\subsection{Minimal position of multicurves}

In this section, we prove that under suitable conditions, the minimal position of two essential simple multicurves on $(\Lambda,P)$ is unique up to a homeomorphism of $\Lambda$ that is the identity on $P$ and isotopic to the identity on $\Lambda-P$. This is surely a known result, but we include a proof since we could not find one in the literature.

\begin{theorem}\label{minimal position is unique surface with boundary} Let $\Lambda$ be a compact surface and $P\subset \partial \Lambda$ a set of marked points on its boundary. Let $\alpha$ and $\beta$ be essential simple multicurves on $(\Lambda,P)$ in minimal position. Assume that $\alpha$ and $\beta$ do not have any shared isotopic components. Let $\alpha'$ and $\beta'$ also be essential simple multicurves in minimal position, belonging to the isotopy classes of $\alpha$ and $\beta$, respectively. Then there exists a homeomorphism $\phi:\Lambda\to \Lambda$ such that $\phi$ is the identity on $P$, $\phi$ is isotopic to the identity on $\Lambda$ with the isotopy fixing $P$, $\phi(\alpha)=\alpha'$ and $\phi(\beta)=\beta'$.
\end{theorem}

\begin{proof} The idea is to encode isotopies from $\alpha$ to $\alpha'$ and $\beta$ to $\beta'$ as surfaces in the $3$-manifold $M=\Lambda\times [0,1]$. First, we may assume $\alpha'=\alpha$ by composing with an appropriate self-homeomorphism of $\Lambda$ fixing $P$. Let $A=\alpha\times[0,1]$, a properly embedded surface in $M$. An isotopy from $\beta$ to $\beta'$ gives a properly embedded surface $B\subset M$ such that $B\cap \Lambda\times \{0\}=\beta$ and $B\cap \Lambda\times \{1\}=\beta'$. Furthermore, $B$ actually lies in $M-P\times[0,1]$ and is isotopic to the surface $\beta\times [0,1]$ in $M-P\times[0,1]$. 

Assume that $A$ and $B$ intersect transversely. Recall that the multicurve $\alpha$ on $(\Lambda,P)$ is the union of arc and closed curve components. Let $\eta$ be an arc component of $\alpha$. Then $\eta\times [0,1]$ is a rectangular component of $A$, with boundaries $\eta\times \{0\}$, $\eta\times \{1\}$ and $\partial \eta\times [0,1]$. We now isotope $B$ so that $B\cap \eta\times [0,1]$ only contains arcs connecting $\eta\times \{0\}$ to $\eta\times \{1\}$.

Apriori, a connected component of $B\cap \eta\times [0,1]$ can also be an arc connecting $\eta\times \{0\}$ or $\eta\times \{1\}$  and $\partial \eta\times [0,1]$, an arc connecting $\eta\times \{0\}$ or $\eta\times \{1\}$ with itself, an arc connecting $\partial \eta\times [0,1]$ with itself, or a closed curve in the interior of the rectangle $\eta\times [0,1]$. We rule out the first three of these possibilities and isotope $B$ to eliminate the last possibility as well.

Suppose $B$ intersects $\eta\times [0,1]$ in an arc connecting  $\eta\times \{0\}$ or $\eta\times \{1\}$  and $\partial \eta\times [0,1]$. Without loss of generality we assume the arc $x$ connects $\eta\times\{0\}$ and $\partial \eta\times [0,1]$. (The case of $\eta\times \{1\}$ is analogous.) There exists an arc $y$ of $\eta\times \{0\}$, with one endpoint in $\partial\eta\times \{0\}$ and the other endpoint an endpoint of $x$, such that $y$ and $x$ are isotopic in $M-P\times [0,1]$. (Isotopic here means that one endpoint of $y$ is fixed, and the other lies in $\partial \eta\times [0,1]$ during the isotopy.) 

By construction $x$ also lies in $B$. Note that $B$ is the union of rectangles and annuli, since it is isotopic to $\beta\times [0,1]$. Since $x$ is an arc with one endpoint on $\Lambda\times \{0\}$ and the other on $\partial \Lambda\times [0,1]$, there exists an arc $y'\subset B$, contained in $\Lambda\times \{0\}$, with one endpoint on $\partial \Lambda \times[0,1]$ and the other an endpoint of $x$, such that $y'$ and $x$ are isotopic in $M-P\times [0,1]$. Now, $y$ is an arc of $\alpha$, and $y'$ is an arc of $\beta$. Since they are both isotopic to $x$ in $M-P\times [0,1]$, $y$ and $y'$ are isotopic in $\Lambda-P$. (Again, isotopy here means that one endpoint of $y$ is fixed, and the other always lies in $\partial \Lambda-P$ during the isotopy.) This means that  $\alpha$ and $\beta$ form a half-bigon, which is a contradiction to the minimal position assumption. 

Similarly, if $B$ intersects $\eta\times [0,1]$ in an arc connecting $\eta\times \{0\}$ or $\eta\times \{1\}$ with itself, then either $\alpha$ and $\beta$ or $\alpha'$ and $\beta'$ form a bigon, which is also a contradiction to the minimal position assumption. If $B$ intersects $\eta\times [0,1]$ in an arc connecting $\partial \eta\times [0,1]$ with itself, then $\alpha$ and $\beta$ have a shared isotopic component which is also a contradiction to the assumptions in the lemma. 

If $B$ intersects $\eta\times [0,1]$ in a closed curve, then we take such a curve innermost $\gamma$ in $\eta\times [0,1]$. The curve $\gamma$ bounds a disk $D$ in $\eta\times [0,1]$ whose interior does not intersect $B$. The curve $\gamma$ must also bound a disk $D'$ in $B$, since $\gamma$ must be nullhomotopic in $M$ and $B$ is isotopic to $\beta\times [0,1]$. The sphere $D\cup D'$ bounds a ball in $M$. Hence, we may isotope $B$ by pusing $D'$ into $D$ to eliminate the intersection $\gamma$. (Our isotopy may possibly eliminate some more closed curve intersections as well.) In this way, we may isotope $B$ to ensure that the only components of $B\cap \eta\times [0,1]$ are arcs connecting $\eta\times\{0\}$ and $\eta\times \{1\}$. 

Next, we consider closed curve components of $\alpha$. Let $\theta$ be such a component. Then $\theta\times [0,1]$ is an annular component of $A$ with boundary $\theta\times\{0\}$ and $\theta\times \{1\}$. Similar to the previous case, we now isotope $B$ so that the only components of $\theta\times [0,1]$ are arcs connecting $\theta\times \{0\}$ to $\theta\times \{1\}$. 

To do this, note that apriori, $B$ can also intersect $\theta\times [0,1]$ in an arc connecting one of $\theta\times \{0\}$ and $\theta\times \{1\}$ with itself, a meridional curve on the annulus, or a null-homotopic curve on the annulus. We rule out the first two possibilities and isotope $B$ to eliminate the last. 

If $B$ intersects $\theta\times [0,1]$ in an arc connecting one of $\theta\times \{0\}$ and $\theta\times \{1\}$ with itself, then either $\alpha$ and $\beta$ or $\alpha$ and $\beta'$ form a bigon, which is a contradiction. If $B$ intersects $\theta\times [0,1]$ in a meridional curve on the annulus, then $\alpha$ and $\beta$ have a shared isotopic closed curve components, which is a contradiction. If $B$ intersects $\theta\times [0,1]$ in a null-homotopic curve on the annulus, we use an innermost disk argument similar to the rectangular case to isotope $B$ and eliminate such intersections.

So, we isotope $B$ (avoiding $P\times [0,1]$) so that every component of $A\cap B$ on $A$ is an arc connecting $\alpha\times \{0\}$ to $\alpha\times \{1\}$. Thus every component of $A\cap B$ on $B$ is now an arc connecting $\beta\times \{0\}$ to $\beta'\times \{1\}$. We straighten these arcs so that each arc is $x\times [0,1]$ for a point $x$ in the interior of $\Lambda$. These arcs divide $B$ into rectangular and annular pieces that we straighten to isotope $\beta'$ into $\beta$ preserving $\alpha$ and fixing $P$. 
\end{proof}

\subsection{Subsurfaces}

\begin{definition}\label{subsurface} A subsurface of $(\Lambda, P)$ is a compact oriented surface $\Omega$ along with an orientation preserving embedding $$\rho:\Omega\to \Lambda.$$ The boundary $\partial \Omega$ has two components. The exterior boundary is defined to be $\rho^{-1}(\partial \Lambda)$. The interior boundary is its complement in $\partial \Omega$ and is denoted by $\bound{\Omega}$. We assume that $\bound{\Omega}$ intersects $\partial \Lambda$ transversely and does not intersect $P$, so that it is an oriented multicurve on $(\Lambda, P)$.
\end{definition} 

If $\Omega\subset (\Lambda,P)$ is a subsurface, the closure of its complement, which we denote by $\Lambda-\Omega$, is also a subsurface of $(\Lambda,P)$. It has the same interior boundary as $\Omega$, with opposite orientation.

\begin{definition} Two subsurfaces $\Omega,\Omega'\subset (\Lambda,P)$ are isotopic if $\bound{\Omega}$ and $\bound{\Omega'}$ are isotopic as oriented multicurves.
\end{definition}

\begin{definition} \label{essential subsurface} A subsurface $\Omega\subset (\Lambda, P)$ is essential if $\bound{\Omega}$ is an essential multicurve on $(\Lambda, P)$. 
\end{definition}

\begin{definition}\label{disk elimination} Let $\Omega\subset (\Lambda,P)$ be a subsurface. A disk addition applied to $\Omega$ is a type of combinatorial operation, which consists of adding or removing a disk from $\Omega$ so that a single inessential component is added to $\bound{\Omega}$. A disk elimination applied to $\Omega$ is a type of combinatorial operation, which consists of adding or removing a disk from $\Omega$ so that a single inessential component is eliminated from $\bound{\Omega}$.
\end{definition}

Note that even though a disk addition may add or remove a disk from $\Omega$, it only adds an inessential component to $\bound{\Omega}$. Analogously for disk eliminations.

\begin{definition}\label{rectangular and annular subsurfaces} A subsurface $\Omega\subset (\Lambda,P)$ is called rectangular if it is homeomorphic to a disk and $$|\bound{\Omega}\cap \partial \Lambda| + |\partial \Omega\cap P|=4.$$ A subsurface $\Omega\subset (\Lambda,P)$ is annular if it is homeomorphic to an annulus, and $$|\bound{\Omega}\cap \partial \Lambda| + |\partial \Omega\cap P|=0.$$ A rectangular component (resp. annular component) of a subsurface of $(\Lambda,P)$ is a connected component of the subsurface that is a rectangular (resp. annular) subsurface of $(\Lambda, P)$. Together, we call rectangular and annular components null components.
\end{definition}

\begin{figure}[h]
%% Creator: Inkscape 1.1.2 (0a00cf5339, 2022-02-04), www.inkscape.org
%% Accompanies image file 'Fig4.pdf' (pdf, eps, ps)
\begingroup%
  \makeatletter%
  \providecommand\color[2][]{%
    \errmessage{(Inkscape) Color is used for the text in Inkscape, but the package 'color.sty' is not loaded}%
    \renewcommand\color[2][]{}%
  }%
  \providecommand\transparent[1]{%
    \errmessage{(Inkscape) Transparency is used (non-zero) for the text in Inkscape, but the package 'transparent.sty' is not loaded}%
    \renewcommand\transparent[1]{}%
  }%
  \providecommand\rotatebox[2]{#2}%
  \newcommand*\fsize{\dimexpr\f@size pt\relax}%
  \newcommand*\lineheight[1]{\fontsize{\fsize}{#1\fsize}\selectfont}%
  \ifx\svgwidth\undefined%
    \setlength{\unitlength}{236.70998771bp}%
    \ifx\svgscale\undefined%
      \relax%
    \else%
      \setlength{\unitlength}{\unitlength * \real{\svgscale}}%
    \fi%
  \else%
    \setlength{\unitlength}{\svgwidth}%
  \fi%
  \global\let\svgwidth\undefined%
  \global\let\svgscale\undefined%
  \makeatother%
  \begin{picture}(1,0.58917259)%
    \lineheight{1}%
    \setlength\tabcolsep{0pt}%
    \put(0,0){\includegraphics[width=\unitlength,page=1]{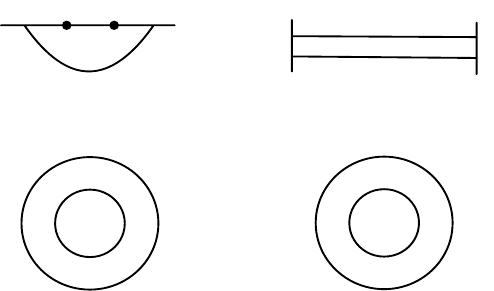}}%
    \put(0.74625035,0.15426136){\makebox(0,0)[lt]{\lineheight{1.25}\smash{\begin{tabular}[t]{l}$\partial \Lambda$\end{tabular}}}}%
    \put(0.97425203,0.47978368){\makebox(0,0)[lt]{\lineheight{1.25}\smash{\begin{tabular}[t]{l}$\partial \Lambda$\end{tabular}}}}%
    \put(0.51596825,0.48431314){\makebox(0,0)[lt]{\lineheight{1.25}\smash{\begin{tabular}[t]{l}$\partial \Lambda$\end{tabular}}}}%
    \put(0.14968785,0.55211458){\makebox(0,0)[lt]{\lineheight{1.25}\smash{\begin{tabular}[t]{l}$\partial \Lambda$\end{tabular}}}}%
  \end{picture}%
\endgroup%

\caption{Types of null components. The marked points on $\partial \Lambda$ are points in $P$.}
\end{figure}
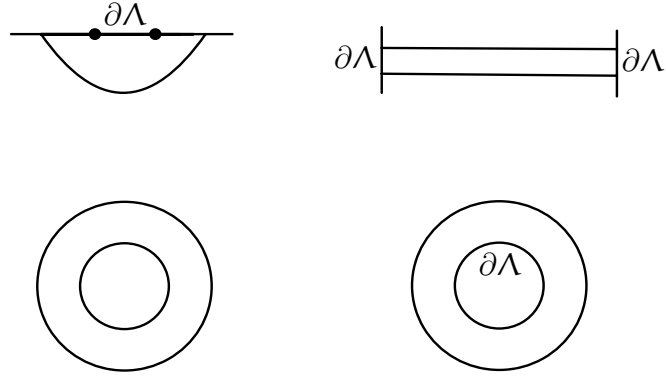

\begin{definition}\label{null component elimination} Let $\Omega\subset (\Lambda, P)$ be a subsurface. A null component addition applied to $\Omega$ is a type of combinatorial operation, which consists of adding or removing a null component from $\Omega$ so that either
\begin{enumerate} \item pair of isotopic arcs or closed curves on $\Lambda$, or 
\item a single inessential closed curve on $\Lambda$ isotopic to a closed curve component of $\partial \Lambda -P$
\end{enumerate}
is added to $\bound{\Omega}$. A null component elimination applied to $\Omega$ is a type of combinatorial operation, which consists of a adding or removing a null component from $\Omega$ so that either (1) or (2) is eliminated from $\bound{\Omega}$.
\end{definition}

Note that a null component addition applied to $\Omega$ only adds a components to $\bound{\Omega}$ (it does not eliminate any). Analogously for null component eliminations. 

\begin{lemma} \label{subsurface of essential part} Let $\Omega\subset (\Lambda,P)$ be a subsurface. There exists a subsurface, denoted $\Omega^{\ess}$, such that $\bound{\Omega^{\ess}}=\bound{\Omega}^{\ess}$.
\end{lemma}

\begin{proof} We apply a sequence of disk eliminations and null component eliminations to $\Omega$, to remove the inessential components of $\bound{\Omega}$. After applying these eliminations, we obtain an essential surface $\Omega^{\ess}$ whose interior boundary is $\bound{\Omega}^{\ess}$.
\end{proof}

\begin{definition} Given a subsurface $\Omega\subset (\Lambda,P)$, let $\widehat{[\Omega]}$ be the unique subsurface obtained by applying iterated null component eliminations to $\Omega^{\ess}$, so that neither $\widehat{[\Omega]}$ nor $\Lambda-\widehat{[\Omega]}$ contain any null components. 
\end{definition}

\begin{definition}\label{equivalence class} Let $\Omega\subset (\Lambda, P)$ be a subsurface. We denote by $[\Omega]$ the equivalence class of  subsurfaces $\Omega'$ such that $\widehat{[\Omega]}=\widehat{[\Omega']}$.
\end{definition}

\subsection{Adjusted Euler characteristic}\label{adjusted Euler char} In this section, we define the adjusted Euler characteristic associated to a subsurface $\Omega$ of $(\Lambda,P)$. It is the standard Euler characteristic of $\Omega^{\ess}$, with an adjustment to encode how $\Omega$ intersects $\partial \Lambda$ and $P$. In the context of adjusted Euler characteristics, we will always assume that the (unadjusted) Euler characteristic of $(\Lambda, P)$ is nonpositive: i.e. $\chi(\Lambda)-|P|\leq 0$. (This is to avoid trivialities involving spheres and unmarked disks.)

\begin{definition} \label{chi 0} Let $\Omega\subset (\Lambda,P)$ be an essential subsurface. We define $$\chi^{\Lambda,P}(\Omega)=\chi(\Omega)-\frac{|\bound{ \Omega}\cap \partial \Lambda|+|\Omega\cap P|}{4}.$$
\end{definition}

\begin{definition} \label{chi 1} Let $\Omega\subset (\Lambda, P)$ be a subsurface. We define $$\chi^{\Lambda,P}(\Omega)=\chi^{\Lambda,P}(\Omega^{\ess}).$$
\end{definition}

We now list some basic properties of $\chi^{\Lambda,P}$.

\begin{lemma}[Nonpositivity] $$\chi^{\Lambda,P}(\Omega)\leq 0,$$ with equality if and only if $\Omega^{\ess}$ is a union of null components.
\end{lemma}

\begin{proof} Suppose $\chi^{\Lambda,P}(\Omega)>0$ for some $\Omega$. We may assume that $\Omega$ is essential. We may also assume $\Omega$ is connected (otherwise, $\chi^{\Lambda,P}$ would be positive for some connected component). This means $\chi(\Omega)>0$, so $\Omega$ is homeomorphic to a disk. Now, $|\bound{\Omega}\cap \partial \Lambda|+|\Omega\cap P|\leq 3$. This means $\Omega$ is inessential, which is a contradiction.

If $\chi^{\Lambda,P}(\Omega)= 0$, then it must be $0$ on every component. To show the second part of the lemma, we may again assume $\Omega$ is connected. Now, either $\chi(\Omega)=2$ and $$|\bound{\Omega}\cap \partial \Lambda|+|\Omega\cap P|=4,$$ in which case $\Omega$ is a rectangular component. Or, $\chi(\Omega)=0$ and $$|\bound{\Omega}\cap \partial \Lambda|+|\Omega\cap P|=0,$$ in which case $\Omega$ is an annular component.
\end{proof}

\begin{lemma}[Additivity]\label{chi of complement} $$\chi^{\Lambda,P}(\Omega)+\chi^{\Lambda,P}(\Lambda-\Omega)=\chi(\Lambda)-\frac{|P|}{4}.$$
\end{lemma}

\begin{remark} Note that $$\chi(\Lambda)-\frac{|P|}{4}=\chi^{\Lambda, P}(\Lambda).$$
\end{remark}

\begin{proof}[Proof of \cref{chi of complement}] Because $(\Lambda-\Omega)^{\ess}=\Lambda-\Omega^{\ess}$, it suffices to prove the lemma assuming $\Omega$ is essential. By additivity of the Euler characteristic, \begin{equation}\label{eq 3.4.1}\chi(\Lambda)=\chi(\Omega)+\chi(\Lambda-\Omega)-\chi(\bound{\Omega}).
\end{equation} Since every connected component of $\bound{\Omega}$ is a closed curve or arc with boundary points on $\partial \Lambda$, \begin{equation}\label{eq 3.4.2} \chi(\bound{\Omega})=\frac{|\bound{\Omega}\cap \partial \Lambda|}{2}.
\end{equation} Because each point of $P$ is in either $\Omega$ or $\Lambda-\Omega$,
$$\chi^{\Lambda,P}(\Omega)+\chi^{\Lambda,P}(\Lambda-\Omega)=\chi(\Omega)+\chi(\Lambda-\Omega)-\frac{|\bound{\Omega}\cap \partial \Lambda|+|\bound{\Lambda-\Omega}\cap \partial \Lambda|+|P|}{4}.$$ Since $\bound{\Omega}=\bound{\Lambda-\Omega}$, the right-hand side is equal to 
$$\chi(\Omega)+\chi(\Lambda-\Omega)-\frac{|\bound{\Omega}\cap \partial \Lambda|}{2}-\frac{|P|}{4}.$$ The lemma now follows from \cref{eq 3.4.1} and \cref{eq 3.4.2}.
\end{proof}

\begin{lemma}[Constant on equivalence class]\label{chi only depends on equivalence class} Let $\Omega_1,\Omega_2\subset (\Lambda, P)$ be subsurfaces. If $[\Omega_1]=[\Omega_2]$, then $$\chi^{\Lambda, P}(\Omega_1)=\chi^{\Lambda, P}(\Omega_2).$$
\end{lemma}

\begin{proof} If $[\Omega_1]=[\Omega_2]$, then $\Omega_1^{\ess}$ is obtained by applying null component additions and eliminations to $\Omega_2^{\ess}$. Each such component contributes $0$ to $\chi^{\Lambda, P}$.
\end{proof}

\begin{lemma}[Monotonicity]\label{monotonicity for chi} Let $\Omega_1,\Omega_2\subset (\Lambda,P)$ be subsurfaces so that $\Omega_1\subset \Omega_2$. Then $$|\chi^{\Lambda,P}(\Omega_1)|\leq |\chi^{\Lambda,P}(\Omega_2)|.$$
\end{lemma}

\begin{proof} If $\Omega_1\subset \Omega_2$, then $\Omega_1^{\ess}\subset \Omega_2^{\ess}$. So we may assume that $\Omega_1$ and $\Omega_2$ are essential. In this case, $\Omega_2-\Omega_1$ is also an essential subsurface of $\Omega$. 

Now, let $$P'=(\bound{\Omega_2}\cap \partial \Lambda) \cup (\Omega_2\cap P)\subset \partial \Omega_2$$ be a set of marked points on $\partial \Omega_2$. By construction, $\chi^{\Omega_2,P'}(\Omega_2)=\chi^{\Lambda,P}(\Omega_2)$. 

Since $\Omega_1\subset \Omega_2$, $$\bound{\Omega_1}\cap \partial \Omega_2=\bound{\Omega_1}\cap \partial \Lambda.$$ Moreover, no points of $\bound{\Omega_2}\cap \partial \Lambda$ lie in $\Omega_1$. Hence 
\begin{align*}\chi^{\Lambda,P}(\Omega_1)&=\chi^{\Omega_2,P'}(\Omega_1)\\&\geq \chi^{\Omega_2,P'}(\Omega_2)\\&= \chi^{\Lambda,P}(\Omega_2)
\end{align*} by additivity and nonpositivity.
\end{proof}

\subsection{Relative adjusted Euler characteristic} In this section, given two subsurfaces of the same surface, we define the adjusted Euler characteristic of one subsurface relative to the other. Let $\Theta$ be a compact connected surface with boundary. 

\begin{definition} \label{boundary points of subsurface} For any essential subsurface $\Lambda\subset \Theta$, we denote by $P(\Lambda)$ the set of marked points $\bound{\Lambda}\cap \partial \Theta$, which lies on $\partial \Lambda$.
\end{definition}

\begin{definition}\label{minimal position of subsurfaces} Let $\Lambda\subset \Theta$ be an essential subsurface. A subsurface $\Omega\subset \Theta$ is in minimal position with $\Lambda$ if $\Omega$ is essential, and $\bound{\Omega}$ is in minimal position with $\bound{\Lambda}$ as multicurves on $\Theta$.

Similarly, let $\zeta$ be an essential multicurve on $\Theta$. A subsurface $\Omega\subset \Theta$ is in minimal position with $\zeta$ if $\Omega$ is essential, and $\bound{\Omega}$ is in minimal position with $\zeta$.
\end{definition}

In order to define the adjusted Euler characteristic of $\Omega$ relative to $\Lambda$, we put $\Omega$ in minimal position with $\Lambda$, then use \cref{chi 1}. 

\begin{definition}\label{chi 2} Let $\Omega\subset \Theta$ be a subsurface. Let $\widetilde{\Omega^{\ess}}$ be a surface isotopic to $\Omega^{\ess}$ in minimal position with $\Lambda$. Note that $\widetilde{\Omega^{\ess}}\cap \Lambda$ is a subsurface of $(\Lambda, P(\Lambda))$. We define $$\chi_{\Theta}^{\Lambda}(\Omega)=\chi^{\Lambda, P(\Lambda)}(\widetilde{\Omega^{\ess}}\cap \Lambda).$$ 
\end{definition}

\begin{remark}\label{minimal position means nonc} Note that if $\Omega^{\ess}$ is in minimal position with $\Lambda$, then $\Omega^{\ess}\cap \Lambda$ is an essential subsurface of $(\Lambda,P(\Lambda))$. If $\Omega^{\ess}\cap \Lambda$ has an inessential boundary component $\eta$, there are several cases to consider. If $\eta$ is a curve around a point in $P(\Lambda)$, then $\eta$ and $\bound{\Lambda}$ form a half-bigon. If $\eta$ bounds a disk along with an arc of $\partial \Lambda$ not containing a point of $P(\Lambda)$, then either $\eta$ forms a bigon with $\bound{\Lambda}$, or $\eta$ is an inessential arc on $\Theta$. The first case is a contradiction to the fact that $\Omega^{\ess}$ and $\Lambda$ are in minimal position. The second case is a contradiction to the fact that $\Omega^{\ess}$ is essential. Finally, $\eta$ cannot be a closed curve bounding a disk, since in this case $\Omega^{\ess}$ would not be essential. 
\end{remark}

\begin{lemma} The quantity $\chi_\Theta^\Lambda(\Omega)$ does not depend on choice of $\widetilde{\Omega^{\ess}}$.
\end{lemma}

\begin{proof} Let $\Omega_1$ and $\Omega_2$ be two subsurfaces of $\Theta$ both isotopic to $\Omega^{\ess}$, both in minimal position with $\Lambda$. By \cref{minimal position means nonc}, it suffices to show that \begin{equation}\label{eq 3}\chi^{\Lambda, P(\Lambda)}(\Omega_1\cap \Lambda)=\chi^{\Lambda, P(\Lambda)}(\Omega_2\cap \Lambda).
\end{equation}

Next, we reduce to the case where no two components of $\bound{\Lambda}$ are isotopic in $\Theta$. If two components of $\bound{\Lambda}$ are isotopic in $\Theta$, then either $\Lambda$ or $\Theta-\Lambda$ contains a null component. By \cref{bigon criterion}, $\bound{\Omega_1}$ and $\bound{\Omega_2}$ intersect $\partial\Lambda$ the same number of times. Applying null component additions or eliminations to $\Lambda$ therefore changes both sides of \cref{eq 3} by the same amount. Hence, we may assume that neither $\Lambda$ nor $\Theta-\Lambda$ contains a null component, which means that no two components of $\bound{\Lambda}$ are isotopic in $\Theta$.

Let $\bound{\Lambda}'$ be the union of components of $\bound{\Lambda}$ which are isotopic to a component of $\bound{\Omega_1}$ (equivalently, $\bound{\Omega_2}$), as multicurves on $\Theta$. We now modify $\Omega_1$ and $\Omega_2$ to form subsurfaces $\Omega_1'$ and $\Omega_2'$. To construct $\Omega_1'$, we isotope $\Omega_1$ so that $\bound{\Lambda}'\subset \bound{\Omega_1}$, and remove all other components of $\bound{\Omega_1}$ isotopic to a component $\bound{\Lambda}'$ by applying null component eliminations to $\Omega_1$. Each such null component either lies in $\Lambda$ or $\Theta-\Lambda$, so $$[\Omega_1'\cap \Lambda]=[\Omega_1\cap \Lambda]$$ as subsurfaces of $\Lambda$. We construct $\Omega_2'$ analogously, so that $$[\Omega_2'\cap \Lambda]=[\Omega_2\cap \Lambda].$$

Now, by construction, $\bound{\Omega_1'}-\bound{\Lambda}'$ and $\bound{\Omega_2'}-\bound{\Lambda}'$ have no components isotopic to any component $\bound{\Lambda}$, and are both in minimal position with $\bound{\Lambda}$. By \cref{minimal position is unique surface with boundary}, the surfaces $\Omega_1'\cap \Lambda$ and $\Omega_2'\cap \Lambda$ are isotopic on $(\Lambda, P(\Lambda))$. So $$\chi^{\Lambda,P(\Lambda)}(\Omega_1'\cap \Lambda)=\chi^{\Lambda,P(\Lambda)}(\Omega_2'\cap \Lambda).$$ \cref{eq 3} now follows from \cref{chi only depends on equivalence class}.
\end{proof}

Like the adjusted Euler characteristic, the relative adjusted Euler characteristic is also nonpositive. Below, we list some more basic properties of the relative adjusted Euler characteristic.

\begin{lemma}[Constant on equivalence class]\label{chi only depends on equivalence class 2} Let $\Omega_1,\Omega_2\subset \Theta$ be subsurfaces. If $[\Omega_1]=[\Omega_2]$, then $$\chi_{\Theta}^{\Lambda}(\Omega_1)=\chi_{\Theta}^{\Lambda}(\Omega_2).$$
\end{lemma}

\begin{proof} If $[\Omega_1]=[\Omega_2]$, then $\Omega_1^{\ess}$ may be obtained by applying null component additions or eliminations to $\Omega_2^{\ess}$. Isotoping as necessary, we may assume each such component also intersects $\Lambda$ in a null component. The lemma follows from \cref{chi only depends on equivalence class}.
\end{proof}

\begin{lemma}[Monotonicity] \label{monotonicity for chi 2} Let $\Omega_1,\Omega_2\subset \Theta$ be subsurfaces such that $\Omega_1\subset \Omega_2$. Then $$|\chi_{\Theta}^{\Lambda}(\Omega_1)|\leq |\chi_{\Theta}^{\Lambda}(\Omega_2)|.$$
\end{lemma}

\begin{proof} We may put $\Omega_1^{\ess}$ and $\Omega_2^{\ess}$ both in minimal position with $\Lambda$ so that $\Omega_1^{\ess}\subset \Omega_2^{\ess}$. The lemma now follows from \cref{monotonicity for chi}. 
\end{proof}

\subsection{Decomposition of surfaces and additivity} In this section, we prove that if $\Theta$ decomposes as a union of subsurfaces, then $\chi^{\Theta}$ is the sum of  the relative adjusted Euler characteristics associated to each of the subsurfaces. 

\begin{lemma}\label{chi is additive} Let $\Theta$ be a connected surface with boundary. Let $\zeta$ be an essential multicurve on $\Theta$ decomposing $\Theta$ into subsurfaces $\Lambda_1,...,\Lambda_k$. Let $\Omega\subset \Theta$ be a subsurface. Then $$\sum_{i=1}^k \chi_{\Theta}^{\Lambda_i}(\Omega)=\chi^{\Theta}(\Omega).$$
\end{lemma}

\begin{remark}\label{chi is additive 2} In particular, substituting $\Omega=\Theta$, we get $$\sum_{i=1}^k \chi^{\Theta}(\Lambda_i)=\chi(\Theta).$$
\end{remark}

\begin{proof}[Proof of \cref{chi is additive}] It suffices to prove the lemma statement for essential subsurfaces $\Omega$ whose interior boundary is in minimal position with $\zeta$. In other words, $\Omega$ is in minimal position with all the $\Lambda_i$. Hence, 
$$\chi^\Theta(\Omega)=\chi(\Omega)-\frac{|\bound{\Omega}\cap \partial \Theta|}{4}$$ and $$\chi_{\Theta}^{\Lambda_i}(\Omega)=\chi(\Omega\cap \Lambda_i)-\frac{|\bound{\Omega\cap \Lambda_i}\cap \partial \Lambda_i|+|\Omega\cap \Lambda_i\cap P(\Lambda_i)|}{4}.$$

Now, 
\begin{align*}\chi(\Omega)&=-\chi(\Omega\cap \zeta)+\sum_{i=1}^k\chi(\Omega\cap \Lambda_i)\\&=-\frac{1}{2}\sum_{i=1}^k \chi(\Omega\cap \bound{\Lambda_i})+\sum_{i=1}^k\chi(\Omega\cap \Lambda_i),
\end{align*} since $\cup_{i=1}^k \bound{\Lambda_i}$ traverses $\zeta$ twice.

The set $\Omega\cap \bound{\Lambda_i}$ is a union of arcs. Its Euler characteristic may be computed by counting the number of endpoints of the arcs. An endpoint of  $\Omega\cap \bound{\Lambda_i}$ is an intersection point of $\partial \Omega$ and $\bound{\Lambda_i}$ i.e. the set $\partial \Omega\cap \bound{\Lambda_i}$. So 
\begin{align*}\frac{1}{2}\sum_{i=1}^k\chi(\Omega\cap \Lambda_i)&=\sum_{i=1}^k\frac{|\partial \Omega\cap \bound{\Lambda_i}|}{4}\\&=\sum_{i=1}^k\frac{|\bound{\Omega}\cap \bound{ \Lambda_i}|}{4}+\frac{|\Omega\cap \partial \Theta\cap \bound{\Lambda_i}|}{4},
\end{align*}
where the last step follows splitting $\partial \Omega$ into the interior and exterior boundary of $\Omega$ viewed as a subsurface of $\Theta$. 

Now, $$|\Omega\cap \partial \Theta\cap \bound{\Lambda_i}|=|\Omega\cap \Lambda_i\cap P(\Lambda_i)|$$ by definition of $P(\Lambda_i)$. Also, $$|\bound{\Omega}\cap \partial \Theta|+\sum_{i=1}^k|\bound{\Omega}\cap \bound{ \Lambda_i}|=\sum_{i=1}^k|\bound{\Omega\cap \Lambda_i}\cap \partial \Lambda_i|.$$ The lemma follows.
\end{proof}

\subsection{Elementary moves on a subsurface} In this section, we introduce certain combinatorial operations on subsurfaces that we call elementary moves.

\begin{definition}\label{surgery on a subsurface} Let $\Omega\subset (\Lambda,P)$ be a subsurface. Let $\delta$ be an embedded arc whose interior is contained in $\Lambda-\bound{\Omega}$ and endpoints in $\bound{\Omega}\cup \partial \Lambda$. We consider a surgery applied to $\Omega$, along $\delta$.

If $\delta$ is an arc with both endpoints on $\bound{\Omega}$, then we call the surgery type 1. If $\delta$ is an arc with one endpoint is on $\bound{\Omega}$ and one endpoint is on $\partial \Lambda$, we call the surgery type 2. (Note that if $\delta$ is an arc with both endpoints on $\partial \Lambda$, then surgering along $\delta$ is equivalent to a null component addition.)
\end{definition}

\begin{definition}\label{boundary elementary move} Let $\Omega\subset (\Lambda,P)$ be a subsurface. A $\partial$-surgery is a surgery along an embedded arc $\delta$ whose interior is contained in $\partial \Lambda-((\bound{\Omega}\cap \partial \Lambda) \cup P)$ and endpoints in $(\bound{\Omega}\cap \partial \Lambda)\cup P$.

If $\delta$ is an arc with both endpoints on $\bound{\Omega}\cap \partial \Lambda$, we call the $\partial$-surgery type 1. If $\delta$ is an arc with one endpoint on $\bound{\Omega}\cap \partial \Lambda$ and one endpoint on $P$, we call the $\partial$-surgery type 2. (Note that if $\delta$ is an arc with both endpoints on $P$, then surgering along $\delta$ is equivalent to a null component addition.)
\end{definition}

Applying a surgery or $\partial$-surgery to a subsurface results in another subsurface. The isotopy class of the new subsurface only depends on the isotopy class of the curve along which surgery is done. A surgery may be inverted by applying another surgery. The inverse of a type 1 surgery is another type 1 surgery. The inverse of a type 2 surgery is a type 1 $\partial$-surgery. The inverse of a type 2 $\partial$-surgery is another type 2 $\partial$-surgery. 

\begin{definition} Let $\Omega\subset (\Lambda,P)$ be a subsurface. An elementary move on $\Omega$ is a surgery of type 1 or 2, $\partial$-surgery of type 1 or 2, disk addition, disk elimination, null component addition or null component elimination.
\end{definition}

\begin{notation} We write ``$\Omega$ and $\Omega'$ are related by a type 2 surgery'' if we may obtain $\Omega'$ by applying a type 2 surgery to $\Omega$. We write analogously for the other elementary moves as well. The order matters for type 2 surgeries, type 1 $\partial$-surgeries, disk additions, disk eliminations, null component additions and null component eliminations.
\end{notation}

Elementary moves give a distance on the set of subsurfaces of $(\Lambda,P)$, and on the set of equivalence classes of subsurfaces of $(\Lambda, P)$.

\begin{definition} Let $\Omega,\Omega'\subset (\Lambda,P)$ be subsurfaces. We define $d(\Omega,\Omega')$ to be the minimum nonnegative integer $k$ such that $\Omega$ and $\Omega'$ are related by a sequence of $k$ elementary moves.
\end{definition}

\begin{definition} Let $\Omega,\Omega'\subset (\Lambda,P)$ be subsurfaces. We define $d([\Omega],[\Omega'])$ to be the minimum nonnegative integer $k$ such that $\Omega$ and $\Omega'$ are related by a sequence of elementary moves, which contains $k$ number of surgeries of type 1 and 2 and $\partial$-surgeries of type 1 and 2. (The sequence may contain any number of disk additions, disk eliminations, null component additions and null component eliminations.) By construction, the distance only depends on the equivalence classes of $\Omega$ and $\Omega'$.
\end{definition}

We now describe the relationship between these two distances.

\begin{lemma} \label{distance between null eliminated surfaces} If $\Omega$ and $\Omega'$ are subsurfaces with $d(\Omega,\Omega')=1$, then $$d(\widehat{[\Omega]},\widehat{[\Omega']})\leq 3.$$
\end{lemma}

\begin{proof} 
Suppose $\Omega$ and $\Omega'$ are related by a surgery or $\partial$-surgery, along an arc $\delta$. Assume $\delta\subset \Omega$. If $\partial \delta$ does not lie on any disk or null component of $\Omega$ or $\Lambda-\Omega$, then $\widehat{[\Omega]}$ and $\widehat{[\Omega']}$ are related by a surgery along $\delta$. If one endpoint of $\partial \delta$ lies on a disk or null component of $\Omega$ or $\Lambda-\Omega$, then $\widehat{[\Omega]}$ and $\widehat{[\Omega']}$ are related by the composition of a disk  or null component addition, and a surgery along $\delta$. If both endpoints of $\partial \delta$ lie on a disk or null component of $\Omega$ or $\Lambda-\Omega$, then $\widehat{[\Omega]}$ and $\widehat{[\Omega']}$ are related by the composition of one or two disk or null component additions, and a surgery along $\delta$. The case wherein $\delta\subset \Lambda-\Omega$ is similar.

If $\Omega$ and $\Omega'$ are related by a disk addition, disk elimination, null component addition, or null component elimination, then $\widehat{[\Omega]}=\widehat{[\Omega']}$, so the lemma follows.
\end{proof}

\begin{cor}\label{distance between null eliminated surfaces 2} If $\Omega$ and $\Omega'$ are subsurfaces with $d([\Omega],[\Omega'])=1$, then $$d(\widehat{[\Omega]},\widehat{[\Omega']})\leq 3.$$
\end{cor}

\begin{proof} For some $\widetilde{\Omega}\in [\Omega]$ and $\widetilde{\Omega'}\in [\Omega']$, $d(\widetilde{\Omega},\widetilde{\Omega'})=1$. The lemma now follows from \cref{distance between null eliminated surfaces}.
\end{proof}

Next, we describe how elementary moves change the adjusted Euler characteristic.

\begin{lemma}\label{elementary move on chi 1} Let $\Omega,\Omega'\subset (\Lambda,P)$ be essential subsurfaces with $d(\Omega,\Omega')=1$. Then $$|\chi^{\Lambda, P}(\Omega)-\chi^{\Lambda,P}(\Omega')|\leq 1.$$
\end{lemma}

\begin{proof} 
It suffices to prove that a disk elimination or addition, surgery or $\partial$-surgery applied to any subsurface $\Omega$ (not necessarily essential) changes the quantity 
$$\chi(\Omega)-\frac{|\bound{ \Omega}\cap \partial \Lambda|+|\Omega\cap P|}{4}
$$ by at most $1$. (Note that null component additions and eliminations do not change this quantity.) 

A disk addition or elimination either adds or a removes a disk from $\Omega$ or $\Lambda-\Omega$. Switching $\Omega$ and $\Omega'$ as necessary, we may assume that it  removes a disk from $\Omega$. In this case $\chi(\Omega)$ decreases by $1$, while the number of points in $(\bound{\Omega}\cap \partial \Lambda)\cup(\Omega\cap P)$ decreases by at most $3$. Hence the relevant quantity changes by at most $1$. 

A surgery of type 1 or 2, or $\partial$-surgery of type 1 or 2, is a surgery along an arc $\delta$ which is contained either in $\Omega$ or $\Lambda-\Omega$. We may assume $\delta\subset \Omega$, switching $\Omega$ and $\Omega'$ as necessary. It suffices to consider the cases of surgeries of type 1 and 2 and $\partial$-surgeries of type 2, since a type 1 $\partial$-surgery is just the inverse of a type 2 surgery.

A type 1 surgery applied to $\Omega$ along an arc contained in $\Omega$ increases $\chi(\Omega)$ by $1$, and does not change any other term. A type 2 surgery applied to $\Omega$ along an arc contained in $(\partial \Lambda-P) \cap \Omega$ increases $$\frac{|\bound{ \Omega}\cap \partial \Lambda|}{4}$$ by $1/2$, but keeps the other terms constant. A type 2 $\partial$-surgery applied to $\Omega$ along an arc contained in $\Omega$ decreases $$\frac{|\Omega\cap P|}{4}$$ by $1/4$, and keeps the other terms constant.
\end{proof}

\begin{lemma}\label{elementary move on chi 2} Let $\Theta$ be a connected surface with boundary. Let $\zeta$ be an essential multicurve on $\Theta$ decomposing $\Theta$ into subsurfaces $\Lambda_1,...,\Lambda_k$. Let $\Omega,\Omega'\subset \Theta$ be essential subsurfaces related by an elementary move. Then $$\sum_{i=1}^k |\chi_{\Theta}^{\Lambda_i}(\Omega)-\chi_{\Theta}^{\Lambda_i}(\Omega')|\leq 1.$$
\end{lemma}

\begin{proof} Either $\Omega\subset \Omega'$ or $\Omega'\subset \Omega$. The lemma follows from \cref{chi is additive}, \cref{elementary move on chi 1} and monotonicity.
\end{proof}

\section{Paths of subsurfaces} \label{paths of subsurfaces}

\subsection{Continuous paths of subsurfaces} \begin{definition}\label{continuous path of subsurfaces} Let $\Lambda$ be a surface with boundary, with $P$ a set of marked points on $\partial \Lambda$. A continuous path of subsurfaces $\Delta=\{\Omega^t\}$ of $(\Lambda,P)$ is a $1$-parameter family of subsurfaces $\Omega^t\subset (\Lambda,P)$, for $t\in [0,1]$, which is an isotopy except at a discrete set of $t$, when there is a single elementary move.

If the elementary move is a surgery or $\partial$-surgery along a $\delta\subset \Lambda-\Omega^t$, we add a neighborhood of this arc to $\Omega^t$ at time $t$. If the elementary move is surgery or $\partial$-surgery along $\delta\subset \Omega^t$, we remove a neighborhood of $\delta$ from $\Omega^t$ for all time $t+\epsilon$, for $\epsilon$ sufficiently small.

If the elementary move is a disk addition, disk elimination, null component addition or null component elimination, a disk or null component is either added to or removed from $\Omega^t$. In the former case, we add the disk or null component to $\Omega^t$ at time $t$. In the latter case, we remove the disk or null component from $\Omega^t$ for all times $t+\epsilon$, for $\epsilon$ sufficiently small.
\end{definition}

Given a continuous path $\Delta=\{\Omega^t\}$ of subsurfaces of $(\Lambda,P)$, we denote by $\Lambda-\Delta$ the continuous path $\{\Lambda-\Omega^t\}$.

\begin{definition} Let $f:(\Lambda,P)\to (\Lambda, P)$ be a homeomorphism. A continuous path of subsurfaces $\Delta=\{\Omega^t\}$ is $f$-twisted if $f(\Omega^0)$ is isotopic to $\Omega^1$.
\end{definition}

\begin{definition}\label{length of continuous path}
Given a continuous path of subsurfaces $\Delta$, its length, denoted $\length(\Delta)$, is the total number of elementary moves in $\Delta$. Its equivalence class length, denoted $\length\left([\Delta]\right)$, is the total number of surgeries of type 1 and 2 and $\partial$-surgeries of type 1 and 2 in $\Delta$.
\end{definition}

\begin{definition} \label{3 dim realization} Associated to any continuous path of subsurfaces $\Delta=\{\Omega^t\}$ is a (topological) $3$-submanifold $\mathcal{L}(\Delta)\subset \Lambda\times [0,1]$, which we call its $3$-dimensional realization. It is defined by $$\mathcal{L}(\Delta)\cap \left(\Lambda\times \{t\}\right)=\Omega^t.$$ The conditions in \cref{continuous path of subsurfaces} ensure that $\mathcal{L}(\Delta)$ is closed.

If $\Delta$ is $f$-twisted, then $\mathcal{L}(\Delta)$ glues to form a closed $3$-submanifold of the mapping torus $M_f$ of $f$. We denote this submanifold by $\mathcal{L}(\Delta)^\vee$. 
\end{definition}

\begin{definition}\label{splitting of paths} A continuous path $\Delta=\{\Omega^t\}$ splits if there are continuous paths $\Delta_1=\{\Omega_1^t\}$ and $\Delta_2=\{\Omega_2^t\}$ such that $\Omega^t=\Omega_1^t\amalg\Omega_2^t$ is the disjoint union. In this case, we write $\Delta=\Delta_1\amalg \Delta_2$.
\end{definition}

Note that if $\Delta=\Delta_1\amalg \Delta_2$, then $$\length(\Delta)=\length(\Delta_1)+\length(\Delta_2)$$ and $$\length([\Delta])=\length([\Delta_1])+\length([\Delta_2]).$$ Now, a continuous path of subsurfaces splits along connected components of its $3$-dimensional realization:

\begin{lemma}\label{splitting and connectedness} Suppose $\mathcal{L}(\Delta)$ is the disjoint union of closed $3$-submanifolds $\mathcal{L}(\Delta)_1$ and $\mathcal{L}(\Delta)_2$. Then $\Delta=\Delta_1\amalg \Delta_2$ such that $\mathcal{L}(\Delta)_1$ and $\mathcal{L}(\Delta)_2$ are the $3$-dimensional realization of $\Delta_1$ and $\Delta_2$, respectively.
\end{lemma}

\begin{proof} Let $$\Omega_1^t=\mathcal{L}(\Delta)_1\cap (\Lambda\times \{t\})$$ and $$\Omega_2^t=\mathcal{L}(\Delta)_2\cap (\Lambda\times \{t\}).$$ Let $\Delta_1=\{\Omega_1^t\}$ and $\Delta_2=\{\Omega_2^t\}$. To show that $\Delta_1$ and $\Delta_2$ are continuous paths of surfaces, it suffices to show that no surgery in $\Delta$ does connects $\Omega_1^t$ and $\Omega_2^t$. This is true since $\mathcal{L}(\Delta)_1$ and $\mathcal{L}(\Delta)_2$ are disconnected.
\end{proof}

Next, we have a version of splitting for twisted paths.

\begin{definition}\label{cyclic splitting} Let $f:(\Lambda,P)\to (\Lambda,P)$ be a homeomorphism. Let $\Delta=\{\Omega^t\}$ be a $f$-twisted continuous path of subsurfaces of $(\Lambda,P)$. Then $\Delta$ splits with respect to $f$ if $$\Delta=\coprod_{n\in \mathbb{Z}} \Delta_n$$ for continuous paths $\Delta_n=\{\Omega^t_n\}$, satisfying $f(\Omega_n^0)=\Omega_{n-1}^1$ for all $n\in \mathbb{Z}$.
\end{definition}

Again, this corresponds to a property of the $3$-dimensional realization. 

\begin{lemma}\label{cyclic splitting on mapping torus} Let $f:(\Lambda,P)\to (\Lambda,P)$ be a homeomorphism, and $\Delta$ an $f$-twisted continuous path of subsurfaces of $(\Lambda,P)$. Then $\Delta$ splits with respect to $f$ if the composition $$\pi_1(\mathcal{L}(\Delta)^\vee)\xrightarrow{i_*} \pi_1(M_f)\xrightarrow{\pi^f_*} \mathbb{Z}$$ is trivial. Here, $i_*$ is the inclusion on $\pi_1$ and $\pi^f_*$ is the projection $\pi^f:M_f\to S^1$ on $\pi_1$. 
\end{lemma}

\begin{proof}
Let $D$ be a closed disk and let $\pi:\mathbb{R}\times D\to M_f$ be the universal covering map from an infinite solid cylinder $\mathbb{R}\times D$ to the mapping torus $M_f$. The composition $$\pi^f_*\circ i_*:\pi_1(\mathcal{L}(\Delta)^\vee)\to \mathbb{Z}$$ is trivial if and only if $\mathcal{L}(\Delta)^\vee$ has a lift $\widetilde{\mathcal{L}(\Delta)^\vee}\subset \mathbb{R}\times D$ such that the projection map $$\pi:\widetilde{\mathcal{L}(\Delta)^\vee}\to \mathcal{L}(\Delta)^\vee$$ is an isomorphism. 

Let us consider the connected components of $\mathcal{L}(\Delta)$. For any such connected component $C$, $\pi^{-1}(C)\subset [n,n+1]\times D$ for some $n\in \mathbb{Z}$. To each connected component, we assign the integer $n\in \mathbb{Z}$. By \cref{splitting and connectedness}, there is a splitting $$\coprod \Delta_n$$ into paths $\Delta_n=\left\{\Omega_n^t\right\}$, such that the components of $\mathcal{L}(\Delta)$ labelled with the interger $n$ lie in $\mathcal{L}(\Delta_n)$.

Now, $$\Omega_{n-1}^{1}=\coprod_{C\text{ component of } \mathcal{L}(\Delta)\text{ labelled } n-1}C\cap (\Lambda \times \{1\}).$$ Each such $\phi(C)$ lies in the boundary of some component $C'$, and by construction $C'$ must be labelled $n$. Hence $f(\Omega_n^0)=\Omega_{n-1}^1$.
\end{proof}

\subsection{Discrete paths of subsurfaces}

\begin{definition}\label{discrete path of subsurfaces} A discrete path of subsurfaces is a sequence $$\Delta=\{\Omega_{-\ell},...,\Omega_k\}$$ of subsurfaces of $(\Lambda,P)$ such that for all $i\in \{-\ell,k-1\}$, the isotopy classes of $\Omega_i$ and $\Omega_{i+1}$ are related by an elementary move. A discrete path is essential if all subsurfaces are esssential. This means that the elementary moves in an essential discrete path must be surgeries, $\partial$-surgeries, and null component additions or eliminations.
\end{definition}

We allow negative indices in discrete paths in order to make some theorem statements cleaner in future sections. As in the case of continuous paths, we denote by $\Lambda-\Delta$ the path $\{\Lambda-\Omega_{-\ell},...,\Lambda-\Omega_k\}$.

\begin{definition} Let $f:(\Lambda,P)\to (\Lambda,P)$ be a homeomorphism. A discrete path $\Delta=\{\Omega_{-\ell},...,\Omega_k\}$ of subsurfaces of $(\Lambda,P)$ is $f$-twisted if $f(\Omega_{-\ell})$ is isotopic to $\Omega_{k}$. 
\end{definition}

\begin{definition}\label{length of discrete path} Given a discrete path $\Delta=\{\Omega_{-\ell},...,\Omega_k\}$, its length, denoted $\length(\Delta)$, is $\ell+k$. Its equivalence class length, denoted $\length([\Delta])$, is the total number of surgeries of type 1 and 2 and $\partial$-surgeries of type 1 and 2 in $\Delta$.
\end{definition}

\begin{definition}\label{discretization} Let $\Delta=\{\Omega_{-\ell},...,\Omega_k\}$ be a discrete path of subsurfaces of $(\Lambda,P)$. Let $\Delta'=\{\Omega^t\}$ be a continuous path of subsurfaces of $(\Lambda,P)$. $\Delta$ is a discretization of $\Delta'$ if 
\begin{enumerate}
\item For all $i\in \{-\ell,...,k\}$ and $t\in \left(\frac{\ell+i}{\ell+k+1},\frac{\ell+1+i}{\ell+k+1}\right)$, $\Omega^t$ is isotopic to $\Omega_i$.
\item At $t=\frac{\ell+i}{\ell+k+1}$, there is an elementary move from a subsurface isotopic to $\Omega_i$, to a subsurface isotopic to $\Omega_{i+1}$.
\end{enumerate}
\end{definition}

Note that if $\Delta$ is a discretization of $\Delta'$, then $$\length(\Delta)=\length(\Delta')$$ and $$\length([\Delta])=\length([\Delta']).$$

\begin{definition}\label{3 dim realization of discrete path} Let $\Delta$ be a discrete path of subsurfaces of $(\Lambda,P)$. A $3$-dimensional realization of $\Delta$ is a $3$-submanifold $$\mathcal{L}(\Delta')\subset \Lambda\times [0,1],$$ where $\Delta'$ is a continuous path of subsurfaces for which $\Delta$ is a discretization. 
\end{definition}

\subsection{Decomposition of surfaces and paths} Let $\Theta$ be a connected surface with boundary. Let $\zeta$ be an essential multicurve on $\Theta$ decomposing $\Theta$ into subsurfaces $\Lambda_1,...,\Lambda_n$. Suppose we have an essential discrete path of subsurfaces on $\Theta$. In this section, we aim to understand to what extent we can decompose our discrete path to get paths of subsurfaces of the $\Lambda_i$s. 

To understand this, we first consider an essential discrete path of length $1$; that is, a pair $\{\Omega,\Omega'\}$ of essential subsurfaces of $\Theta$ whose isotopy classes are related by an elementary move. We may put both $\Omega$ and $\Omega'$ in minimal position with $\zeta$. We may consider writing a sequence $\{\Omega\cap \Lambda_i,\Omega'\cap \Lambda_i\}$ for each $\Lambda_i$, however, it is not necessarily a discrete path. In fact, although $d(\Omega,\Omega')=1$, $$d(\Omega\cap \Lambda_i,\Omega'\cap \Lambda_i)$$ (as subsurfaces of $\Lambda_i$) may be arbitrarily large. So we cannot decompose our length $1$ discrete path of subsurfaces of $\Omega$ into bounded length paths of subsurfaces of the $\Lambda_i$. However, we can bound the equivalence class length of the components in the decomposition. 

\begin{notation} Let $\Omega\subset \Theta$ be an essential subsurface in minimal position with $\zeta$. Consider $\Omega$ as a surface with marked points on its boundary given by $P(\Omega)$ and $\zeta\cap \partial\Omega$. Let $\delta$ be a properly embedded arc on $\Omega$ (resp. $\Theta-\Omega$). We say that $\delta$ is in minimal position with $\zeta$ if $\delta$ and $\zeta\cap \Omega$ do not bound any bigons or half-bigons  as multicurves on $(\Omega,P(\Omega))$ (resp. $(\Theta-\Omega, P(\Theta-\Omega))$). In this case, we denote by $\Omega\pm \delta$ the subsurface of $\Theta$ obtained by surgering $\Omega$ along $\delta$.  
\end{notation}

\begin{lemma}\label{minimal isotopy exists} Let $\Delta=\{\Omega_{-\ell},...,\Omega_k\}$ be an essential discrete path of subsurfaces of $\Theta$. Assume the $\Omega_i$ are in minimal position with $\zeta$. Then for each $1\leq i\leq n$ there is a continuous path $\Delta_i=\{\Omega^t_i\}$ of subsurfaces of $\Lambda_i$, satisfying $$\Omega^{(j+\ell)/(\ell+k+1)}_i=\Omega_j\cap \Lambda_i$$ and $$\sum_{i=1}^k \length([\Delta_i])\leq 6\length(\Delta).$$ Moreover, the $3$-dimensional realizations $\mathcal{L}(\Delta_i)\subset \Lambda_i\times [0,1]$ glue to form a $3$-submanifold of $\Theta\times [0,1]$, which is isotopic to a $3$-dimensional realization of $\Delta$.
\end{lemma}

\begin{proof} It suffices to prove the lemma for paths of length $1$, i.e. $\Delta=\{\Omega,\Omega'\}$; concatenation gives the lemma in the general case. It also suffices to prove the lemma when $\Omega $ and $\Omega'$ are related by a surgery of type 1 and 2 and null component additions, as the lemma then follows for type 1 $\partial$-surgeries and null component eliminations by symmetry (note that type 2 $\partial$-surgeries do not exist in this setting, since $\Theta$ does not have marked point on its boundary). Now, $\Omega$ and $\Omega'$ are related by a surgery along $\delta$ (note that a null component addition may be viewed as a surgery along an arc with both points endpoints on $\partial \Theta$). We may ensure $\delta$ is in minimal position with $\zeta$ by simply isotoping $\delta$ to eliminate any bigons or half-bigons with $\zeta$. This does not change the isotopy class of the surgered surface. 

We now construct the $1$-parameter families $\Omega_i^t$. First, we do a surgery to $\Omega$ along $\delta$. This corresponds to at most $2$ surgeries on the $\Omega\cap \Lambda_i$, as well as some null component additions to some of the $\Omega\cap \Lambda_i$. (We may do these operations one-by-one to ensure that the result is a path.) So far, the contribution to the total equivalence class length so far is at most $2$.

Second, we put $\Omega\pm \delta$ in minimal position with $\zeta$. To do this, we eliminate maximal bigons or half-bigons between $\bound{\Omega \pm\delta}$ and $\zeta$. Here, maximal means that the arc of $\bound{\Omega\pm \delta}$ is maximal. Since $\delta$ does not form any bigons or half-bigons with $\zeta$, any maximal bigon or half-bigon must contain an endpoint of one of the two arcs in $\bound{\Omega\pm \delta}$ parallel to $\delta$. So there are at most four such bigons or half-bigons. When we eliminate a maximal bigon, we apply disk eliminations or null component eliminations to some of the $\Omega \pm \delta \cap \Lambda_i$, and do at most one type 1 $\partial$-surgery to one of the $\Omega \pm \delta \cap \Lambda_i$. When we eliminate a maximal half-bigon, we apply disk eliminations or null component eliminations to some of the $\Omega \pm \delta\cap \Lambda_i$, and do at most one type 2 $\partial$-surgery to one of the $\Omega \pm \delta\cap \Lambda_i$. Hence the contribution of these operations to the total equivalence class length is at most $4$.

By \cref{minimal position is unique surface with boundary}, once $\bound{\Omega\pm \delta}$ is in minimal position with $\zeta$, it is isotopic to $\bound{\Omega'}$ except for shared isotopic components. So we apply some null component additions and eliminations to finish constructing the set of paths. The total contribution to the sum of the equivalence class lengths is at most $6$. By construction, the $3$-dimensional realizations of the components in the decomposition glue to form a $3$-submanifold isotopic to a $3$-dimensional realization of $\{\Omega,\Omega'\}$. 
\end{proof}

\section{Surfaces in a product $3$-manifold} \label{Surfaces in a product}

\subsection{Preliminaries} 

Let $W$ be a compact connected surface with boundary. Let $N=W\times [0,1]$, so that $$\partial N=W\times \{0\}\cup \partial W\times [0,1]\cup W\times \{1\}.$$ In this section, we consider properly embedded surfaces in $N$. We develop a theory of compressibility and $\partial$-compressibility of properly embedded surfaces that recognizes the decomposition of $\partial N$ above.

\begin{definition}\label{strong isotopy} Let $F,F'\subset N$ be properly embedded surfaces. We say that $F$ and $F'$ are strongly isotopic if there is an isotopy of $N$, preserving $\partial W\times \{0\}$ and $\partial W\times \{1\}$, taking $F$ to $F'$.
\end{definition}

\begin{definition}\label{bigon removing isotopy} Let $F\subset N$ be a properly embedded surface. Let $D\subset \partial W\times [0,1]$ be a bigon formed by $\partial F\cap (\partial W\times [0,1])$ and $\partial W\times \{0,1\}$. A bigon removing isotopy of $F$ is an isotopy of $F$ along $D$, so that the bigon is eliminated. 
\end{definition}

Recall the definition of a compressing disk:

\begin{definition}\label{compression} Let $F\subset N$ be a properly embedded surface. A compressing disk $D$ for $F$ is a disk $D\subset N$ such that 
\begin{enumerate} 

\item $\partial D\subset F$,

\item the interior of $D$ is disjoint from $F$, and 

\item $\partial D$ is essential in $F$.

\end{enumerate}
$F$ is incompressible if $F$ does not admit any compressing disk, and no connected component of $F$ is a sphere.
\end{definition} 

We now introduce some relevant definitions of $\partial$-compressions.

\begin{definition}\label{boundary compression along annular boundary region} Let $F\subset N$ be a properly embedded surface. A $\partial$-compressing disk along $\partial W\times[0,1]$ for $F$ is an embedded disk $D\subset N$ such that $D\cap F=a$ and $D\cap \partial N=b$ are arcs in $\partial D$ with disjoint interiors, satisfying:

\begin{enumerate}
\item $a\cup b=\partial D$,

\item $a\cap b=\partial a=\partial b\subset \partial F$,

\item $b\subset \partial W\times [0,1]$ and

\item $a$ does not cobound a disk in $F$ with another arc in $\partial F\cap (\partial W\times [0,1])$. 
\end{enumerate}

We call $F$ $\partial$-compressible along $\partial W\times [0,1]$ if $F$ is a disk that cobounds a ball along with a disk in $\partial W\times [0,1]$, or $F$ admits a $\partial$-compressing disk along $\partial W\times[0,1]$, in which case a $\partial$-compression to $F$ is a surgery along such a disk. Otherwise, $F$ is $\partial$-incompressible along $\partial W\times [0,1]$.
\end{definition}

\begin{definition}\label{straight boundary} A properly embedded surface $F\subset N$ has straight boundary if $$\partial F\cap (\partial W\times [0,1])=P\times [0,1],$$ where $P$ is a discrete set of points in $\partial W$. 
\end{definition}

A properly embedded surface $F\subset N$ naturally comes with some marked points on its boundary, and it will be useful to keep track of these points.

\begin{definition}\label{marked points on surface} Given a properly embedded surface $F\subset N$, we let $$Q(F)=F\cap (\partial W\times \{0,1\}),$$ a set of marked points lying on $\partial F$.
\end{definition}

\begin{lemma}\label{number of bigon removing isotopies} Let $F\subset N$ be a properly embedded surface such that $F$ is $\partial$-incompressible along $\partial W\times[0,1]$, and $F$ does not intersect $\partial W\times [0,1]$ in a longitudal curve. After at most $Q(F)/2$ number of bigon removing isotopies, $F$ is strongly isotopic to a surface with straight boundary.
\end{lemma}

\begin{proof} If $F$ is not already isotopic to a surface with straight boundary, then $\partial F$ intersects $\partial W\times \{0,1\}$ in a bigon. A bigon removing isotopy removes the bigon, and reduces $Q(F)$ by $2$. 
\end{proof}

Next, we have another modified definition of $\partial$-incompressibility, suited for surfaces with straight boundary in $N$.

\begin{definition}\label{boundary compression along W} Let $F\subset N$ be a surface with straight boundary. A $\partial$-compressing disk along $W$ for $F$ is an embedded disk $D\subset N$ such that $D\cap F=a$ and $D\cap \partial N=b$ are arcs in $\partial D$ with disjoint interiors, satisfying:

\begin{enumerate}
\item $a\cup b=\partial D$,

\item $a\cap b=\partial a=\partial b\subset \partial F$,

\item $b\subset W\times\{0\}\cup W\times \{1\}$ and

\item $a$ does not cobound a disk in $F$ with another arc in $\partial F\cap (W\times\{0\}\cup W\times \{1\})$. 
\end{enumerate}

We call $F$ $\partial$-compressible along $W$ if $F$ is a disk that cobounds a ball along with a disk in $W\times\{0,1\}$, or $F$ admits a $\partial$-compressing disk along $W$, in which case a $\partial$-compression to $F$ is a surgery along such a disk. Otherwise, $F$ is $\partial$-incompressible along $W$.
\end{definition}

\begin{remark} Condition (4) in \cref{boundary compression along annular boundary region} and condition (4) in \cref{boundary compression along W} are equivalent to the statement that $a$, viewed as a multicurve on $(F,Q(F))$, is essential. The definition of standard $\partial$-incompressibility is constructed to so that a $\partial$-compression simplifies the topology of the surface. Similarly, \cref{boundary compression along annular boundary region} and \cref{boundary compression along W} is constructed so that a $\partial$-compression along $W$ simplifies the topology of $(F,Q(F))$.
\end{remark}

We quantify this observation (for $\partial$-compressions along $W$) in the following lemma.

\begin{lemma} \label{number of compressions} Let $F\subset N$ be an incompressible surface with straight boundary. After at most $-\chi(F)+|Q(F)|$ number of $\partial$-compressions along $W$, $F$ becomes a surface that is incompressible and $\partial$-incompressible along $W$. 
\end{lemma}

\begin{proof} Let $D\cap F=a$. When compressed, $F$ is cut along $a$, but $F\cap (\partial W\times \{0,1\})$ does not change. Thus $\chi(F)-|Q(F)|$ (which is a nonpositive integer) is increased by $1$. The lemma follows.
\end{proof}

Finally, we note that our various definitions of compressibility are compatible.

\begin{lemma}\label{compatibility of compression definitions} 
Let $F\subset N$ be a properly embedded surface, and let $F'\subset N$ the result of a bigon removing isotopy (or compression, or $\partial$-compression along $\partial W\times [0,1]$, or $\partial$-compression along $W$) applied to $F$. If $F$ is incompressible (or $\partial$-incompressible along $\partial W\times [0,1]$, or $\partial$-incompressible along $W$), so is $F'$.
\end{lemma}

\begin{proof} The only nontrivial part is to show that if $F'$ is the result of a bigon removing isotopy applied to $F$, and $F$ is $\partial$-incompressible along $\partial W\times[0,1]$, so is $F'$. Suppose the contrary. Let $D$ be a $\partial$-compressing disk along $\partial W\times[0,1]$. Strongly isotoping $F'$ as necessary, we may ensure that $D$ avoids a small neighborhood of $\partial W\times \{0,1\}$. Reversing the bigon removing isotopy, we may strongly isotope $F$ so that the bigon is contained in a small neighborhood of $\partial W\times \{0,1\}$. This means $D$ is also a $\partial$-compressing disk along $\partial W\times [0,1]$ for $F$, which is a contradiction.
\end{proof}

\subsection{Classification of incompressible surfaces}
In the remainder of this section, we classify surfaces with straight boundary in $N$ that are incompressible and $\partial$-incompressible along $W$.

\begin{theorem}\label{incompressible structure theorem} Let $N=W\times [0,1]$ and $F\subset N$ be a connected  surface with straight boundary that is incompressible and $\partial$-incompressible along $W$. Then $F$ is strongly isotopic to $\alpha\times [0,1],$ where $\alpha$ is a curve or arc on $W$. 
\end{theorem}

\begin{cor}\label{incompressible structure theorem 2} Let $N=W\times [0,1]$ and $F\subset N$ be a surface with straight boundary that is incompressible and $\partial$-incompressible along $W$. Then $F$ is strongly isotopic to $\alpha\times [0,1],$ where $\alpha$ is a multicurve on $W$.
\end{cor}

\begin{proof} First, we show that that each connected component of $F$ is strongly isotopic to a product as in the statement of \cref{incompressible structure theorem}. Assume the contrary. Then it admits a compressing disk or $\partial$-compressing disk $D$ along $W$. 

Suppose it admits a compressing disk $D$. Apriori, $D$ may not be a compressing disk for $F$ because it may intersect other components of $F$. However, isotoping $D$ so that it intersects $F$ minimally and choosing an innermost closed curve of $F\cap D$ gives a compressing disk for $F$. This is a contradiction.

Analogously, suppose the connected component of $F$ admits a $\partial$-compressing disk $D$ along $W$. Again, $D$ may intersect other components of $F$, hence may not be a $\partial$-compressing disk along $W$ for $F$. Isotoping $D$ so that it intersects $F$ minimally and choosing an outermost bigon on $D$ gives a $\partial$-compressing disk along $W$ for $F$, which is a contradiction.

Therefore, each connected component of $F$ is strongly isotopic to a product. To show that $F$ itself is strongly isotopic to a product, we do induction on the number of components. The base case is when $F$ is connected, which is already proved. For the induction step, we first strongly isotope $F$ so that one component, $F_0$, is a product. Then, we cut along $F_0$ and use the induction hypothesis to conclude that the other components of $F$ may also be simultaneously strongly isotoped to a product.
\end{proof}

The rest of this section will be to prove \cref{incompressible structure theorem}. To do this, we do induction on the topological type of $W$.

\subsection{Base case of induction} In this section, we prove the following:

\begin{lemma}\label{induction base case for incompressible structure theorem} Let $W$ be a disk, and $N=W\times [0,1]$. Let $F\subset N$ be a connected surface with straight boundary, and suppose $F$ is incompressible and $\partial$-incompressible along $W$. Then $F$ is strongly isotopic to $\alpha\times [0,1]$,  where $\alpha$ is a simple arc on $W$ with boundary points on $\partial W$.
\end{lemma}

If $F$ is incompressible, it must also be $\pi_1$-injective. Since $N$ is homeomorphic to a ball, this means $F$ is homeomorphic to a disk. Note also that $\partial N$ is homeomorphic to $S^2$. The boundary $\partial F$ is a curve on $\partial N$.

Note that $\partial F$ intersects $W\times \{0\}$ an even number of times. If $\partial F$ does not intersect $\partial W\times\{0\}$ at all, then since $\partial F$ is straight, it does not intersect $\partial W\times \{1\}$ either. Thus $\partial F\subset W\times \{0\}\cup W\times \{1\}$. In particular, $\partial F$ bounds a disk in either $W\times \{0\}$ or $W\times \{1\}$, and there is an innermost component of $\partial F$ in this disk. This innermost component bounds a disk which is a $\partial$-compressing disk for $F$ along $W$. Hence, $\partial F$ intersects $\partial W\times \{0\}$ at least twice.  

Now, $\partial F$ is a closed curve on $\partial N$, hence bounds disk $D\subset \partial N$. There is a bigon formed by $\partial F$ and $\partial W\times \{0\}$ that is contained in $D$. Let $\theta$ be the arc of $\partial F$ associated to this bigon. The endpoints of $\theta$ are two points $x\times \{0\}$ and $y\times\{0\}$ lying on $\partial W\times \{0\}$. Because $F$ is admissible, $x\times [0,1]$ and $y\times [0,1]$ are also arcs of $\partial F$. Now, consider the edge $e$ of $W\times\{1\}$ between $x\times \{1\}$ and $y\times \{1\}$ that is in $D$. Since $F$ is $\partial$-incompressible along $W$, $e$ forms a bigon with an arc of $\partial F$ in $W\times \{1\}$. Thus, $\partial F$ intersects both $\partial W\times \{0\}$ and $\partial W\times \{1\}$ exactly twice. Therefore we may strongly isotope $F$ so that $F=\alpha\times[0,1]$, for an arc $\alpha\subset W$ with endpoints on $\partial W$. 

\subsection{Induction step}

Let $W$ be a compact surface with boundary and $N=W\times [0,1]$. Let $\eta$ be an essential arc on $W$ with boundary on $\partial W$, so that cutting $W$ along $\eta$ produces a connected surface $W'$. We prove the following induction step for \cref{incompressible structure theorem}.

\begin{lemma}\label{induction step of incompressible structure theorem} Suppose the statement of \cref{incompressible structure theorem} holds for $W'$. Then it also holds for $W$.
\end{lemma}

Let $F\subset N$ be a surface with straight boundary, incompressible and $\partial$-incompressible along $W$. Isotoping $F$ as necessary, we may assume that $\partial F$ does not intersect $\partial \eta\times [0,1]$. 

Now, further isotope $F$ so that the number of connected components of $F\cap (\eta \times [0,1])$ is minimal. First, we use an innermost loop/outermost bigon argument to claim that $F$ intersects $\eta\times [0,1]$ only along horizontal arcs (i.e. an arc connecting $\eta\times \{0\}$ to $\eta\times \{1\}$).

To see this, note that by assumption, $F$ only intersects $\partial (\eta\times [0,1])$ along $\eta\times \{0\}$ and $\eta\times \{1\}$ (since $F$ does not intersect $\partial \eta\times [0,1]$). If our claim were false, $F\cap (\eta\times [0,1])$ would contain either an innermost closed curve on $\eta\times [0,1]$, or an outermost bigon along $\eta\times \{0\}$ and $\eta\times \{1\}$.

Suppose $F\cap (\eta\times [0,1])$ contains an innermost closed curve $\theta$ on $\eta\times [0,1]$. Then $\theta$ bounds a disk $D$ on $\eta\times [0,1]$. Since $D$ cannot be a compressing disk for $F$, $\theta$ also bounds a disk $D'$ on $F$. Since $\theta$ was chosen to be innermost, $D'$ and $D$ do not intersect. Hence $D\cup D'$ is an embedded sphere in $N$. Since $N$ is irreducible, the sphere $D\cup D'$ is the boundary of a ball $B$. Since the interior of $D$ is disjoint from $F$ and $F$ is incompressible, the interior of $B$ is disjoint from $F$. Pushing $D'$ into $D$ along the ball isotopes $F$ decreasing the number of connected components of $F\cap (\eta\times [0,1])$. This is a contradiction to minimality. 

Next, suppose $F\cap (\eta\times [0,1])$ contains an outermost bigon along $\eta\times \{0\}$. Let $\xi$ be the arc of $F\cap \eta\times [0,1]$ associated to this bigon. The interior of the bigon is a disk $D$ in $\eta\times [0,1]$ satisfying $D\cap F=\xi$. By assumption, $D$ cannot be a $\partial$-compressing disk along $W$ for $F$. Hence the arc $\xi$ of $F\cap (\eta\times [0,1])$ in this bigon cobounds a disk $D'$ in $F$ along with another arc in $\partial F\cap W\times \{0\}$. Since $D\cap F=\xi$, $D'$ is disjoint from $D-\theta$. Thus $D\cup D'$ is an embedded disk with boundary on $W\times \{0\}$. The curve $\partial (D\cup D')$ is an embedded curve on $W\times\{0\}$ and must be inessential, $W\times \{0\}$ is incompressible in $N$. Therefore, $\partial (D\cup D')$ also bounds an embedded disk $D''$ in $W\times \{0\}$. The sphere $D\cup D'\cup D''$ is an embedded sphere in $N$ and by irreducibility bounds a ball $B$. Because $\partial B\cap F=D'$ and $F$ is incompressible, $B\cap F=D'$. Isotoping $F$ by pushing $D'$ along $B$ into $D$ decreases the number of connected components of $F\cap (\eta\times [0,1])$. This is also a contradiction to minimality. A similar argument shows that there are no outermost bigons along $\eta\times \{1\}$, either. 

Therefore, $F$ only intersects $\eta\times [0,1]$ in horizontal arcs. Straightening, we may assume $F\cap (\eta \times [0,1])= P\times [0,1]$ for a discrete set of points $P\subset \eta$. Let $N'=W'\times [0,1]$. Then $N'$ can be obtained by cutting $N$ along $\eta\times [0,1]$. Cutting $F$ along  $\eta\times [0,1]$ gives a properly embedded surface $F'\subset N'$ with straight boundary.

\begin{lemma} The surface $F'\subset N$ is incompressible and $\partial$-incompressible along $W'$. 
\end{lemma}

\begin{proof} Since $F$ is incompressible in $N$, $F'$ is incompressible in $N'$. It remains to show that $F'$ is $\partial$-incompressible along $W'$. Suppose the contrary. Then there exists a $\partial$-compressing disk $D$ for $F'$ along $W'$. By construction, $D$ does not intersect $\eta\times [0,1]$, therefore, $D$ is also a disk in $N$. Since $D\cap F'$ is essential in $F'$ so $D\cap F$ is also essential in $F$. Thus $D$ is a $\partial$-compressing disk for $F$ along $W$, which is a contradiction.
\end{proof}

By our assumption that \cref{incompressible structure theorem} holds for $W'$, $F'$ is strongly isotopic to $C'\times [0,1]$ where $C'$ is a multicurve on $W'$. This means $C'$ glues along the two components of $\eta$ on $W'$ to give a multicurve $C\subset W$. Since the strong isotopy can be chosen to be the identity on $F'\cap (\eta\times [0,1])$, it also glues to give a strong isotopy of $F$ to $C\times [0,1]$.

\section{A dictionary between $3$-submanifolds and paths of subsurfaces} \label{Dictionary}

\subsection{$3$-submanifolds of knot complements and products} Let $K$ be any knot and $M=S^3-N(K)$. 

\begin{notation}\label{3-submanifold of M} Let $L$ be a $3$-submanifold of $M$. The boundary of $L$ splits into two components: $\partial L\cap \partial M$, which we call the exterior boundary; and the closure of its complement in $\partial L$, which we call the interior boundary. Where the interior boundary is relevant, we shall denote a $3$-submanifold of $M$ by a pair $(L,S)$ where $L$ is the $3$-submanifold and $S$ is its interior boundary. 

If $L\subset M$ is a $3$-submanifold, we denote the closure of its complement by $M-L$.
\end{notation}

\begin{definition} A $3$-submanifold $(L,S)\subset M$ has spherical boundary if $\genus(S)=0$ (so $S$ is the union of some spheres with boundary components).
\end{definition}

Next, let $K\subset S^3$ be a fibered knot. Let $W$ be the Seifert surface that is the fiber of the fiber bundle associated to $K$. Let $f:W\to W$ be a monodromy map for the fiber bundle. (We do not assume that $f$ is the identity on $\partial W$.) Let $M=S^3-N(K)$. Fix a fiber $W\subset M$, which is a properly embedded, incompressible and $\partial$-incompressible surface. Let $N=W\times [0,1]$, obtained from $M$ by cutting along $W$.

\begin{notation}\label{3-submanifold of N} Let $L$ be a $3$-submanifold of $N$. The boundary of $L$ splits into two components: $\partial L\cap \partial N$, which we call the exterior boundary; and the closure of its complement in $\partial L$, which we call the interior boundary. We shall denote a $3$-submanifold of $N$ by a pair $(L,S)$ where $L$ is the $3$-submanifold and $S$ is its interior boundary.

Again, for $3$-submanifolds $L\subset N$, we denote the closure of the complement by $N-L$.
\end{notation}

\begin{definition}\label{submanifold of N} Given a $3$-submanifold $(L,S)\subset M$, we let $(L^\wedge,S^\wedge)$ be the $3$-submanifold of $N$ obtained by cutting $L$ and $S$ along $W$. Given a $3$-submanifold $(L,S)\subset N$ such that $f(L\cap (W\times \{0\}))$ is isotopic to $L\cap (W\times \{1\})$, we let $(L^\vee,S^\vee)\subset M$ be the $3$-submanifold formed by gluing $L$ and $S$ along $W\times \{0,1\}$. 
\end{definition}

\begin{definition} \label{tunnel component} A connected component of a properly embedded surface $S\subset N$ is a tunnel if it is strongly isotopic to the interior boundary of a neighborhood of an arc or curve on $W\times \{0,1\}$.
\end{definition}

\begin{definition} A $3$-submanifold $(L,S)\subset N$ is admissible if the following conditions are satisfied:

\begin{enumerate}

\item $L\cap (W\times \{0\})$ and $L\cap (W\times \{1\})$ are essential subsurfaces of $W$.

\item $S$ is an incompressible surface.

\item $S$ is $\partial$-incompressible along $\partial W\times [0,1]$.

\item $S$ does not intersect $\partial W\times [0,1]$ in a longitudal curve.

\end{enumerate}
\end{definition}

\begin{lemma}\label{surgery on 3-manifold gives elementary move on surface} Let $(L,S)\subset N$ be an admissible $3$-submanifold. \begin{enumerate} 

\item Assume $S$ has straight boundary. Let $S'$ be a surface formed by applying a $\partial$-compression along $W$ to $S$. There is an admissible $3$-submanifold $L'\subset N$ whose interior boundary is $S'$, which satisfies the following property. If the $\partial$-compressing disk is along $W\times \{0\}$, $L'\cap (W\times \{0\})$ and $L\cap (W\times \{0\})$ are related by a type 1 surgery. If the $\partial$-compressing disk is along $W\times \{1\}$, $L'\cap (W\times \{1\}$ and $L'\cap (W\times \{1\})$ are related by a type 1 surgery. 

\item Assume $S$ has no tunnel components. Let $S'\subset N$ be a surface formed by applying a bigon removing isotopy to $S$. There is an admissible $3$-submanifold $L'$ whose interior boundary is $S'$, which satisfies the following property. If the bigon is along $W\times \{0\}$, $L'\cap (W\times \{0\})$ and $L\cap (W\times \{0\})$ are related by a type 2 surgery. If the bigon is along $W\times \{1\}$, $L'\cap (W\times \{1\}$ and $L\cap (W\times \{1\})$ are related by a type 2 surgery. 

\item Let $S'\subset N$ be a subsurface formed by removing a tunnel component from $S$. There is an admissible $3$-submanifold $L'$ whose interior boundary is $S'$, which satisfies the following property. If the tunnel component is along $W\times \{0\}$, $L'\cap (W\times \{0\})$ and $L\cap (W\times \{0\})$ are related by a null component addition. If the tunnel component is along $W\times \{1\}$, $L'\cap (W\times \{1\})$ and $L\cap (W\times \{1\})$ are related by a null component addition. 
\end{enumerate}
\end{lemma}

\begin{proof} First, we prove statement (1). Assume that the $\partial$-compressing disk is along $W\times \{0\}$ (the other case is analogous). The disk is either contained in $L$ or $N-L$. In the former case, we remove a neighborhood of the disk from $L$ to obtain $L'$. In the latter case, we add a neighborhood of the disk to $L$ to obtain $L'$. In either case, $L\cap (W\times \{0\})$ and $L'\cap (W\times \{0\})$ are related by a type 1 surgery. This means $L'\cap (W\times\{0\})$ and  $L\cap (W\times \{0\})$ are also related by a type 1 surgery. We now show that $L'$ is admissible. The only nontrivial part is to show that $L'\cap (W\times \{0\})$ is an essential subsurface of $W$. This is a consequence of condition (4) in \cref{boundary compression along W}. 

Next, we prove statement (2). Assume that the bigon is along $W\times \{0\}$ (the other case is analogous). The bigon is contained either in $L$ or $N-L$. If it is contained in $L$, we remove a neighborhood of the bigon from $L$ to obtain $L'$. If the bigon is contained in $N-L$, we add a neighborhood of the bigon to $L$ to obtain $L'$. In either case, $L\cap (W\times \{0\})$ and $L'\cap (W\times \{0\})$ are related by a type 1 $\partial$-surgery. This means $L'\cap (W\times\{0\})$ and  $L\cap (W\times \{0\})$ are related by a type 2 surgery. To show that $L'$ is admissible, again, the only nontrivial part is to show that $L'\cap (W\times \{0\})$ is an essential subsurface of $W$. This is true since $S$ does not contain any tunnel components and is $\partial$-incompressible along $\partial W\times[0,1]$. 

Finally, we prove statement (3). Assume that the tunnel is along $W\times \{0\}$ (the other case is analogous). The tunnel is the boundary of the neighborhood of an arc or closed curve on $W\times \{0\}$. This neighborhood is contained in $L$ or $N-L$. If it is contained in $L$, we remove it from $L$ to obtain $L'$. If it is contained in $N-L$, we add it to $L$ to obtain $L'$. In either case, $L\cap (W\times \{0\})$ and $L'\cap (W\times \{0\})$ are related by a null component elimination. Equivalently, $L'\cap (W\times\{0\})$ and  $L\cap (W\times \{0\})$ are related by a null component addition. Since $L'\cap (W\times \{0\})$ is formed by applying a null component elimination to $L\cap (W\times \{0\})$ (an essential subsurface of $W$), $L'\cap (W\times \{0\})$ is an essential subsurface as well. Therefore $L'$ is admissible.
\end{proof}

\subsection{Structure theorem for $3$-submanifolds of a fibered knot complement}

In this section, we relate certain types of $3$-submanifolds of $M$ to discrete paths of subsurfaces of $W$. First, we describe the conditions on the $3$-submanifold we need for our dictionary to hold.

\begin{definition}\label{admissible submanifold of M} A $3$-submanifold $(L,S)\subset M$ is admissible if $(L^\wedge,S^\wedge)\subset N$ is admissible, and $S^\wedge$ does not contain any tunnel components.
\end{definition}

\begin{definition} \label{isotopy of 3-submanifolds} Let $(L,S), (L',S')\subset N$ be $3$-submanifolds. We say that $(L,S)$ and $(L',S')$ are isotopic if $S$ and $S'$ are strongly isotopic, and the ambient isotopy takes $L$ to $L'$. 
\end{definition}

A $3$-dimensional realization of an $f$-twisted discrete path gives a $3$-submanifold of $M$. The main theorem of this section explains how to go in the other direction; that is, how to obtain an $f$-twisted discrete path of subsurfaces from an admissible $3$-submanifold. 

\begin{theorem} \label{discrete path associated to 3-submanifold} Let $(L,S)\subset M$ be an admissible $3$-submanifold. There is an $f$-twisted essential discrete path $$\Delta=\{\Omega_{-\ell},...,\Omega_0,...,\Omega_k\}$$ of subsurfaces of $W$, satisfying the following conditions.
\begin{enumerate} \item $\length(\Delta)\leq 4|\chi(S)|$. 
\item $L^\wedge$ is isotopic to a $3$-dimensional realization of $\Delta$, as $3$-submanifolds of $N$.
\item For all $i>0$, $\Omega_{i-1}$ and $\Omega_i$ are related by a surgery of type 1 or 2, or null component addition; and $\Omega_{-i+1}$ and $\Omega_{-i}$ are  related by a surgery of type 1 or 2, or null component addition. 
\end{enumerate}
\end{theorem}

\begin{proof} Consider $S\cap W$ as a multicurve on $S$. First, we show that it is essential, using an innermost bigon/outermost disk argument. If it contains a closed nullhomotopic loop, then choose an innermost such loop $\xi$ on $S$. The loop $\xi$ bounds a disk on $S$ that does not intersect $W$. Because $K$ is fibered, $W\subset M$ is an incompressible surface, hence $\xi$ is inessential on $W$ also. This contradicts our assumption that $(L,S)$ is in minimal position with $W$. If a component of $S\cap W$ along with an arc of $\partial S$ forms a bigon on $S$, then we take an outermost such bigon $D$. Then $D$ is disk in $M$ with boundary arc $\theta$ on $W$ and $\partial D-\theta\subset \partial M$. If $\theta$ were essential in $W$, then $D$ would be a $\partial$-compressing disk for $W\subset M$, which is a contradiction. Therefore, $\theta$ is inessential in $W$, which is a contradiction our minimal position assumption. 

We now study the connected components of $S^\wedge$; let $\mathcal{C}$ be the set of connected components of $S^\wedge$. On $S$, these connected components are simply the connected components of the complement of $W\cap S$. Since $W\cap S$ is essential, for all $R\in \mathcal{C}$ we have $\chi^S(R)\leq 0$. (By admissibility, the complement of $W\cap S$ cannot have any spherical components. Otherwise, $S^\wedge$ would not be an incompressible surface in $N$.) Subdivide $\mathcal{C}$ into $$\mathcal{C}_{\chi=0}=\{R \text{ is a connected component of } S^\wedge|\chi^S(R)=0\}$$ and $$\mathcal{C}_{\chi< 0}=\{R \text{ is a connected component of } S^\wedge|\chi^S(R)< 0\}.$$

Recall from \cref{marked points on surface} that $Q(R)=R\cap (\partial W\times \{0,1\})$. Then $Q(R)$ consists of the points where $R$ intersect $\partial S$. So $$\chi^S(R)=\chi(R)-\frac{|Q(R)|}{4}.$$ 

\begin{lemma}\label{classification of 0 components} If $\chi^S(R)=0$, then $R$ is isotopic to $\alpha\times [0,1]$, for a properly embedded arc or closed curve $\alpha\subset W$.
\end{lemma}

\begin{proof} If $\chi^S(R)=0$, then $R$ is a disk that intersects $\partial W\times \{0,1\}$ four times. There are three cases to consider.

\textbf{Case 1:} $\partial R$ intersects $\partial W\times \{0\}$ twice and $\partial W\times \{1\}$ twice. In this case, $\partial R\cap W\times \{0\}$ and $\partial R\cap W\times \{1\}$ must be isotopic arcs, so the lemma follows.

\textbf{Case 2:} $\partial R$ intersects $\partial W\times \{0\}$ four times. In this case, $\partial R\cap W\times \{0\}$ contains two connected components, which must be isotopic. Thus $S^\wedge$ contains a tunnel component, which is a contradiction.

\textbf{Case 3:} $\partial R$ intersects $\partial W\times \{1\}$ four times. Similar to case 2 above, we have a contradiction again.
\end{proof}

By \cref{chi is additive 2}, 
\begin{align*}\sum_{{R\in \mathcal{C}_{\chi<0}}} \chi(R)-\frac{|Q(R)|}{4}&=\sum_{R\in \mathcal{C}_{\chi<0}} \chi^S(R)\\&=\chi(S).
\end{align*} Since $\chi(R)\leq 0$, this means $$\sum_{i=1}^k \chi(R)-|Q(R)|\geq 4\chi(S).$$ 

We now do a series of bigon removing isotopies, tunnel removals, and $\partial$-compressions along $W$ to $S^\wedge$. First, we do a series of bigon removing isotopies and tunnel removals to obtain surfaces $S^\wedge_0=S^\wedge,S^\wedge_1,...,S^\wedge_m$. Each surface is obtained by applying a bigon removing isotopy to the previous one, if the previous one does not have any tunnel components; or the removal of a tunnel, if the previous one does have a tunnel component. Finally, $S^\wedge_m$ is strongly isotopic to a surface with straight boundary. Second, we do a series of $\partial$-compressions along $W$ starting from $S^\wedge_m$ to obtain surfaces $S^\wedge_{m+1},...,S^\wedge_n$. Each surface is obtained by applying a $\partial$-compression along $W$ to the previous one, and $S^\wedge_n$ is incompressible and $\partial$-incompressible along $W$. By \cref{number of bigon removing isotopies} and \cref{number of compressions}, $n\leq 4|\chi(S)|$. 

By \cref{surgery on 3-manifold gives elementary move on surface}, we have a sequence $L^\wedge_i\subset N$ of $3$-submanifolds, each having $S^\wedge_i$ as its interior boundary. Using this sequence, we construct a two-sided sequence of subsurfaces of $W$, as follows. Because $S^\wedge_n$ is incompressible and $\partial$-incompressible along $W$, by \cref{incompressible structure theorem}, it is a product. We let $$\Omega_0=L^\wedge_n\cap (W\times \{0\})=L^\wedge_n\cap (W\times \{1\}).$$ At each stage, to obtain $L^\wedge_i$ from $L^\wedge_{i+1}$, we remove a disk neighborhood along $W\times \{0\}$ or $W\times \{1\}$. In the first case, we add the subsurface $L^\wedge_i\cap (W\times \{0\})$ to the negative side of the sequence. In the second case, we add the subsurface $L^\wedge_i\cap (W\times \{1\})$ to the positive side of the sequence. In this way, we have a sequence $$\Delta=\{\Omega_{-\ell},...,\Omega_0,...,\Omega_k\}$$ of subsurfaces of $W$, where $\ell+k=n\leq 4|\chi(S)|$. By \cref{surgery on 3-manifold gives elementary move on surface} again, $\Delta$ is an essential discrete path of subsurfaces satisfying condition (3) in the theorem statement. By construction, $L^\wedge$ is isotopic to a $3$-dimensional realization of $\Delta$. Since $L^\wedge$ is constructed by taking a subsurface $L\subset M$ and cutting along $W$, $$f(L^\wedge\cap (W\times \{0\}))=L^\wedge\cap (W\times \{1\}.$$ This means $f(\Omega_{-\ell})$ is isotopic to $\Omega_k$, hence $\Delta$ is an $f$-twisted path.
\end{proof}

\subsection{$3$-submanifolds of a torus knot complement} \label{dictionary for torus knots} Let $M=S^3-N(T_{p,p+1})$. In this section, we further extend the dictionary between $3$-submanifolds of $M$ and sequences of subsurfaces of $\Sigma_{p,p+1}$.

Recall from \cref{Preliminaries} that the multicurve $\cup_{i,j}\alpha_{i,j}$ decomposes $\Sigma_{p,p+1}$ into the $U_i$s and $V_j$s, which are all disks. In this language, the monodromy map $f:\Sigma_{p,p+1}\to \Sigma_{p,p+1}$ sends $U_i$ to $U_{i+1(\modulo p)}$ and $V_j$ to $V_{j+1(\modulo p+1)}$, by \cref{monodromy of torus knot}.

Let $U$ be the closed disk and $P_U$ be a set of $2p$ marked points on $\partial U$. Similarly, let $V$ also be a copy of the closed disk with $P_V$ a set of $2(p+1)$ marked points on $\partial V$. Considering $U_i$ and $V_j$ as subsurfaces of $\Sigma_{p,p+1}$, note that each $(U_i,P(U_i))$ (resp. $(V_j,P(V_j))$) is homeomorphic to $(U,P_U)$ (resp. $(V,P_V)$) (see \cref{boundary points of subsurface}). We identify each $(U_i,P(U_i))$ with $(U,P_U)$ so that the monodromy map is an isomorphism $U_i\to U_{i+1}$ for $1\leq i\leq p-1$, and the monodromy map $U_p\to U_1$ is a homeomorphism shifting the marked points by two in the direction of the orientation. We denote this map on $U$ by $$\phi:(U,P_U)\to (U,P_U).$$ Similarly, we identify each $(V_j,P(V_j))$ with $(V,P_V)$ so that the monodromy map is an isomorphism $V_j\to V_{j+1}$ for $1\leq j\leq p$, and the monodromy map $V_{p+1}\to V_1$ is a homeomorphism shifting the marked points by two in the direction opposite the orientation. We denote this map on $V$ by $$\psi:(V,P_V)\to (V,P_V).$$ 

\begin{construction}\label{U and V sequences} Let $\Delta=\{\Omega_{-\ell},...,\Omega_k\}$ be an $f$-twisted essential discrete path of subsurfaces of $\Sigma_{p,p+1}$. Put each subsurface in minimal position with the $U_i$s and $V_j$s. By \cref{minimal isotopy exists}, associated to $\Delta$ there are continuous paths of surfaces $\Delta_{U_i}=\{\Omega_{U_i}^t\}$ of $(U,P_U)$  and $\Delta_{V_j}=\{\Omega_{V_j}^t\}$ of $(V,P_V)$, satisfying the following conditions:
\begin{enumerate}
\item $\Omega_{U_i}^0=\Omega_{-\ell}\cap U_i$ and $\Omega_{U_i}^1=\Omega_k\cap U_i$.
\item $\Omega_{V_j}^0=\Omega_{-\ell}\cap V_j$ and $\Omega_{V_j}^1=\Omega_k\cap V_j$.
\item $$\sum_{i=1}^p\length([\Delta_{U_i}])+\sum_{j=1}^{p+1}\length([\Delta_{V_j}])\leq 6\length(\Delta).$$
\item The $3$-dimensional realizations $\mathcal{L}(\Delta_{U_i})$ and $\mathcal{L}(\Delta_{V_j})$ glue to form a $3$-submanifold $\mathcal{L}\subset N$ isotopic to a $3$-dimensional realization of $\Delta$.
\end{enumerate} 

Since $\Delta$ is $f$-twisted, $\Omega_k\cap U_i$ is isotopic to $\Omega_{-\ell}\cap U_{i+1}$ for all $1\leq i\leq p$. Therefore, the $\Delta_{U_i}$s may be concatenated (and rescaled in the time direction) to form a $\phi$-twisted continuous path of subsurfaces $$\Delta_U=\{\Omega_U^t\}$$ of $(U,P_U)$ such that 
\begin{enumerate}
\setcounter{enumi}{4}
\item $\Omega_U^0=\Omega_{-\ell}\cap U_1$, 
\item $\Omega_U^1=\Omega_k\cap U_p$, and
\item $\length([\Delta_U])\leq 6\length(\Delta)$. 
\end{enumerate} 

Analogously, there is a $\psi$-twisted continuous path of subsurfaces $$\Delta_V=\{\Omega_V^t\}$$ of $(V,P_V)$ such that 
\begin{enumerate}
\setcounter{enumi}{7}
\item $\Omega_V^0=\Omega_{-\ell}\cap V_1$, 
\item $\Omega_V^1=\Omega_k\cap V_{p+1}$, and
\item $\length([\Delta_V])\leq 6\length(\Delta)$. 
\end{enumerate} 
\end{construction}

Let $T$ be the union of multicurve $\cup_{i,j}\alpha_{i,j}$ over all of the fibers $\Sigma_{p,p+1}$ in $M$. By \cref{torus in fiber bundle}, $T$ bounds two solid tori in $M$. One, denoted $T_U$, is the union of the $U_i$s over all the fibers. The other, denoted $T_V$, is the union of the $V_j$s over all the fibers.

\begin{theorem} \label{cyclic splitting and homology} Let $\Delta$ be an $f$-twisted essential discrete path of subsurfaces of $\Sigma_{p,p+1}$. Let $\Delta_U$, $\Delta_V$ be continuous paths of subsurfaces, as in \cref{U and V sequences}. There is a $3$-submanifold $\mathcal{L}$, isotopic to $\mathcal{L}(\Delta)$, such that:
\begin{enumerate} 
\item If the inclusion $$H_1(\mathcal{L}^\vee\cap T_U)\to H_1(T_U)$$ is trivial, then $\Delta_U$ splits with respect to $\phi$. 
\item If the inclusion $$H_1(\mathcal{L}^\vee\cap T_V)\to H_1(T_V)$$ is trivial, then $\Delta_V$ splits with respect to $\psi$.
\item If the inclusion $$H_1((M-\mathcal{L}^\vee)\cap T_U)\to H_1(T_U)$$ is trivial, then $U-\Delta_U$ splits with respect to $\phi$.
\item If the inclusion $$H_1((M-\mathcal{L}^\vee)\cap T_V)\to H_1(T_V)$$ is trivial, then $V-\Delta_V$ splits with respect to $\psi$.
\end{enumerate}
\end{theorem}

\begin{proof}

Recall that $\phi(\mathcal{L}(\Delta_U)\cap (U\times \{0\}))$ and $\mathcal{L}(\Delta_U)\cap (U\times \{1\})$ are isotopic. So gluing, we obtain a submanifold $\mathcal{L}(\Delta_U)^\vee\subset T_U$. By statement (4) in \cref{U and V sequences}, $$\mathcal{L}^\vee\cap T_U=\mathcal{L}(\Delta_U)^\vee.$$ If the inclusion $$H_1(\mathcal{L}^\vee\cap T_U)\to H_1(T_U)\simeq \mathbb{Z}$$ is trivial, then the inclusion $$\pi_1(\mathcal{L}^\vee\cap T_U)\to \pi_1(T_U)\simeq \mathbb{Z}$$ is also trivial, since it factors through the former map by Hurewicz. By \cref{cyclic splitting on mapping torus}, $\Delta_U$ splits with respect to $\phi$. This proves statement (1). The proofs of statements (2), (3) and (4) are analogous.
\end{proof}

\section{Intersection of the Seifert surface of $T_{p,p+1}\# K$ with embedded balls} \label{Topological lemma section}

Let $\Sigma$ be a minimal genus Seifert surface for $T_{p,p+1}\#K$. The goal of this section is to show that if $\Sigma$ is cut by an embedded ball into pieces bounded below and above in a suitable sense, then the boundary of the ball must intersect the knot many times. The precise formulation of the statement involves the relative adjusted Euler characteristic, and is in \cref{statement of topological lemma}.

\subsection{Preliminaries} 
The Seifert surface $\Sigma$ decomposes as $\Sigma=\Sigma_{p,p+1}\# \Sigma_K$, where $\Sigma_K$ is a minimal genus Seifert surface for $K$. Fix representatives of $\alpha_{i,j}$ on $\Sigma_{p,p+1}$. Fix representatives of $\alpha_{i,j}$ and $\bound{\Sigma_{p,p+1}}$ on $\Sigma$. Below, we define some notational shorthands for adjusted Euler characteristics associated to $\Sigma$ and $\Sigma_{p,p+1}$. 

\begin{notation}\label{char on Sigma p} Let $\Omega\subset \Sigma_{p,p+1}$ be a subsurface. We let $$\chi^{p,p+1}(\Omega)=\chi^{\Sigma_{p,p+1}}(\Omega).$$
\end{notation}

\begin{notation}\label{char on Sigma} Let $\Omega\subset \Sigma$ be a subsurface. We let $$\chi^{p,p+1}_{\Sigma}(\Omega)=\chi^{\Sigma_{p,p+1}}_{\Sigma}(\Omega).$$
\end{notation}

\begin{remark}\label{char on Sigma p vs Sigma} In \cref{char on Sigma p}, we do not have marked points associated to $\Sigma_{p,p+1}$. In \cref{char on Sigma}, we do have marked points associated to $\Sigma_{p,p+1}$. So if $\Omega\subset \Sigma$ is an essential subsurface in minimal position with $\Sigma_{p,p+1}$, 
$$\left|\chi^{p,p+1}(\Omega\cap \Sigma_{p,p+1})-\chi^{p,p+1}_{\Sigma}(\Omega)\right|\leq \frac{1}{2},$$
accounting for the discrepancy in marked points.
\end{remark}

\subsection{Statements of the main theorems in this section}\label{statement of topological lemma}

\begin{theorem}\label{topological lemma for connected sum} Let $\beta\subset \mathbb{R}^3$ be an embedding of $T_{p,p+1}\#K$. Let $\Sigma\subset \mathbb{R}^3$ be a minimal genus Seifert surface with $\partial \Sigma=\beta$. Let $S$ be an embedded $S^2$ in $\mathbb{R}^3$ intersecting $\Sigma$ transversely. Suppose $$\frac{p(p-1)}{10}\leq |\chi^{p,p+1}_{\Sigma}(\interior(S)\cap \Sigma)| \leq \frac{9p(p-1)}{10}.$$ Then $$|S\cap \beta|\geq \frac{p}{10^4}.$$
\end{theorem}

\cref{topological lemma} follows as a corollary.

\begin{proof}[Proof of \cref{topological lemma}] We apply \cref{topological lemma for connected sum} to the case where $K$ is the unknot. Since $$\frac{p(p-1)}{20}\leq \genus(\interior(S)\cap \Sigma), \genus(\exterior(S)\cap \Sigma)\leq \frac{9p(p-1)}{20},$$ we have that $$\frac{p(p-1)}{10}\leq |\chi^{p,p+1}_{\Sigma}(\interior(S)\cap \Sigma)|\leq \frac{9p(p-1)}{10},$$ so the conditions in the statement of \cref{topological lemma for connected sum} are satisfied.
\end{proof}

\cref{topological lemma for connected sum} follows from the following theorem.

\begin{theorem} \label{general topological lemma} Let $M$ be $S^3-N(T_{p,p+1})$. Let $\Sigma_{p,p+1}\subset M$ be a properly embedded oriented surface of genus $p(p-1)/2$ such that $\partial \Sigma_{p,p+1}\subset\partial M$ is a longitude. Let $(L,S)\subset M$ be a $3$-submanifold with spherical boundary. Suppose $$\frac{p(p-1)}{10}\leq |\chi^{p,p+1}(L\cap \Sigma_{p,p+1})|\leq \frac{9p(p-1)}{10}.$$ Then $$|\chi(S)|\geq \frac{p}{10^4}.$$
\end{theorem}

In the rest of the section, we prove \cref{topological lemma for connected sum} and \cref{general topological lemma}.

\subsection{\cref{general topological lemma} implies \cref{topological lemma for connected sum}} Let $\beta$ be an embedding of $T_{p,p+1}\# K$ in $\mathbb{R}^3$. Let $\Sigma$ be a minimal genus Seifert surface with $\partial \Sigma=\beta$. Let $S$ be a sphere intersecting $\Sigma$ transversely and $L$ be the embedded ball it bounds. Assuming $$\frac{p(p-1)}{10}\leq |\chi^{p,p+1}_\Sigma(\interior(S)\cap \Sigma)| \leq \frac{9p(p-1)}{10}$$ and $$|S\cap \beta|\leq \frac{p}{10^4},$$ we will show a contradiction. 

Compactifying, we may assume $\beta$, $\Sigma$ and $S$ lie in $S^3$. Let $H$ be a sphere in $S^3$ intersecting $\Sigma$ in one connected arc, decomposing it into $\Sigma_{p,p+1}=\interior(H)\cap \Sigma$ and $\Sigma_K=\exterior(H)\cap \Sigma$.

First, we isotope $S$ so that $S\cap \Sigma$ and $H\cap \Sigma$ are in minimal position as multicurves on $\Sigma$. If $S\cap \Sigma$ and $H\cap \Sigma$ form a bigon, we may choose an innermost such bigon and isotope $S$ along the bigon to eliminate it. This operation does not change $\chi_{\Sigma}^{p,p+1}(L\cap \Sigma)$. 

Assume $S\cap \Sigma$ and $H\cap \Sigma$ are in minimal position. Now, $H\cap S$ is a disjoint union of embedded closed curves. Surgering $S$ (and adding or removing disks from $L$) to eliminate components of $H\cap S$ that do not intersect $\Sigma$, we may assume that every component of $H\cap S$ intersects $\Sigma$. These surgeries happen away from $\Sigma$, so again $\chi_{\Sigma}^{p,p+1}(L\cap \Sigma)$ is unchanged.

After doing these operations, we have a surface $S$ which is a disjoint union of spheres, bounding $L$, a $3$-manifold in $S^3$. The pair $(L',S')=(L\cap \interior(H),S\cap \interior(H))$ is a $3$-submanifold of $M=\interior(H)-N(\beta)$, which is homeomorphic to $S^3-N(T_{p,p+1})$.

\begin{lemma}\label{Euler characteristic of S'} $$|\chi(S')|\leq \frac{p}{10^4}.$$
\end{lemma}

\begin{proof} Consider $H\cap S$ as a union of embedded curves on $S$. Take an innermost such curve $\alpha$, which bounds a disk $D$ on $S$. We claim that $$D\cap \beta\neq \emptyset.$$

To see this claim, suppose the contrary. Since $\alpha$ is an innermost curve of $H\cap S$ on $S$, the interior of $D$ does not intersect $H$. So, $D$ lies in either $\interior(H)$ or $\exterior(H)$. In either case, by minimality of the Seifert surface $\Sigma_{p,p+1}$ or $\Sigma_K$, either $D\cap \Sigma$ forms a bigon on $\Sigma$ with $H\cap \Sigma$, or $D$ does not intersect $\Sigma$. Both cases are ruled out by our operations modifying $S$. Hence we have a contradiction.

Now, suppose $$|\chi(S')|\geq \frac{p}{10^4}.$$ Every component of $S-(H\cap S)$ is a sphere with some holes. Of these components, only disks have positive Euler characteristic. By supposition, $S-(H\cap S)$ contains at least $p/10^4$ disks. By our claim, each disk intersects $\beta$. So $$|S\cap \beta|\geq \frac{p}{10^4},$$ which is a contradiction. 
\end{proof}

By construction, $$\frac{p(p-1)}{10}\leq |\chi^{p,p+1}_{\Sigma}(L'\cap \Sigma)| \leq \frac{9p(p-1)}{10}.$$ Since $L'\cap \Sigma$ is in minimal position with $\Sigma_{p,p+1}$, by \cref{char on Sigma p vs Sigma}, $$\frac{p(p-1)}{10}\leq |\chi^{p,p+1}(L'\cap \Sigma_{p,p+1})| \leq \frac{9p(p-1)}{10}.$$ This contradicts \cref{general topological lemma}. 

\subsection{Proof of \cref{general topological lemma}: topology}

Assume that $$\frac{p(p-1)}{10}\leq \chi^{p,p+1}(L\cap \Sigma_{p,p+1})\leq \frac{9p(p-1)}{10}.$$ Assume also that $$|\chi(S)|\leq \frac{p}{10^4}.$$ We will show a contradiction. We start by modifying $(L,S)$ so that is admissible.

First, we modify $(L,S)$ so that $S^\wedge\subset N$ is an incompressible surface. To do this, note that any compressing disk for $S^\wedge$ is a compressing disk $D$ for $S$ such that $D\subset M-\Sigma_{p,p+1}$. Compressing $S$ along $D$ (and removing a neighborhood of $D$ from $L$) decreases $|\chi(S)|$ but does not change $\chi^{p,p+1}(L\cap \Sigma_{p,p+1})$. By doing these operations we may assume that $S^\wedge\subset N$ is incompressible.

Second, we modify $(L,S)$ so that $L\cap \Sigma_{p,p+1}$ is an essential subsurface of $\Sigma_{p,p+1}$. Note that $S\cap \Sigma_{p,p+1}$ is the interior boundary of $L\cap \Sigma_{p,p+1}$ as a subsurface of $\Sigma_{p,p+1}$. If $S\cap \Sigma_{p,p+1}$ contains an inessential closed curve component on $\Sigma_{p,p+1}$, we take an innermost such component and compress $S$ along the disk in $\Sigma_{p,p+1}$ bounded by it. Such a compression decreases $|\chi(S)|$ but does not change $\chi^{p,p+1}(L\cap \Sigma_{p,p+1})$. If $S\cap \Sigma_{p,p+1}$ contains an inessential arc component on $\Sigma_{p,p+1}$, we take an outermost bigon on $\Sigma_{p,p+1}$, and $\partial$-compress $S$ along the bigon. Again, such a compression decreases $|\chi(S)|$ but does not change $\chi^{p,p+1}(L\cap \Sigma_{p,p+1})$.

Third, we modify $(L,S)$ so that $S^\wedge$ is $\partial$-incompressible along $\partial \Sigma_{p,p+1}\times[0,1]$. If $D$ is a $\partial$-compressing disk along $\partial \Sigma_{p,p+1}\times[0,1]$ for $S^\wedge$, we compress $S^\wedge$ along $D$ and glue the ends back together to modify $S$. We modify $L$ by adding or removing a neighborhood of $D$. This operation does not increase $|\chi(S)|$ or change $\chi^{p,p+1}(L\cap \Sigma_{p,p+1})$.

Fourth, we eliminate tunnel components of $S^\wedge$. If $S^\wedge$ contains a tunnel component, we take an outermost such, and isotope $S$ by pushing the tunnel component through a rectangular or annular subsurface of $\Sigma_{p,p+1}$. This operation does not change $\chi(S)$. It also does not change $\chi^{p,p+1}(L\cap \Sigma_{p,p+1})$, since null component eliminations do not change the adjusted Euler characteristic.

Finally, the Euler characteristic bound on $S$ implies that $S$ does not intersect $\partial M$ in a longitudal curve. 

Thus, we may assume that $(L,S)\subset M$ is an admissible $3$-submanifold. By \cref{discrete path associated to 3-submanifold}, there is an $f$-twisted discrete path $\Delta$, with $\length(\Delta)\leq 4|\chi(S)|$, such that $L^\wedge$ is isotopic to a $3$-dimensional realization of $\Delta$. By \cref{U and V sequences} and \cref{cyclic splitting and homology}, there are continuous paths $\Delta_U$ and $\Delta_V$, whose $3$-dimensional realizations glue to form a $3$-submanifold $\mathcal{L}$ isotopic to a $3$-dimensional realization of $\Delta$, satisfying conditions (1) and (2) in the theorem statement. Denote by $\mathcal{S}$ the interior boundary of $\mathcal{L}$.

Then $\mathcal{L}^\vee\subset M$ is also a $3$-submanifold with spherical boundary. It satisfies
\begin{equation}\label{eq 4}\frac{p(p-1)}{10}\leq \chi^{p,p+1}(\mathcal{L}^\vee\cap \Sigma_{p,p+1})\leq \frac{9p(p-1)}{10}
\end{equation} and 
\begin{equation}\label{eq 5}
|\chi(\mathcal{S}^\vee)|\leq \frac{p}{10^4}.
\end{equation}

We will prove the following lemma in \cref{proof of topological lemma: combinatorics}.

\begin{lemma}\label{combinatorial lemma for cyclic splitting} Either $\Delta_U$ does not split with respect to $\phi$, or $\Delta_V$ does not split with respect to $\psi$. Similarly, either $U-\Delta_U$ does not split with respect to $\phi$, or $V-\Delta_V$ does not split with respect to $\psi$.
\end{lemma}

Assuming this lemma for now, let 
\begin{equation*}
\begin{cases} I_U(\mathcal{L}^\vee)&:H_1(\mathcal{L}^\vee\cap T_U)\to H_1(T_U)\\ I_V(\mathcal{L}^\vee)&:H_1(\mathcal{L}^\vee\cap T_V)\to H_1(T_V)\\ I_U(M-\mathcal{L}^\vee)&:H_1((M-\mathcal{L}^\vee)\cap T_U)\to H_1(T_U)\\ I_V(M-\mathcal{L}^\vee)&:H_1((M-\mathcal{L}^\vee)\cap T_V)\to H_1(T_V)
\end{cases}
\end{equation*} be the inclusion maps on homology. By \cref{cyclic splitting and homology} and \cref{combinatorial lemma for cyclic splitting}, either $I_U(\mathcal{L}^\vee)$ or $I_V(\mathcal{L}^\vee)$ is nontrivial. Also, either $I_U(M-\mathcal{L}^\vee)$ or $I_V(M-\mathcal{L}^\vee)$ is nontrivial. We split into four cases.

\textbf{Case 1.} $I_U(\mathcal{L}^\vee)$ and $I_U(M-\mathcal{L}^\vee)$ are nontrivial. By Mayer-Vietoris, $$H_1(\mathcal{S}^\vee\cap T_U)\to H_1(\mathcal{L}^\vee\cap T_U)\oplus  H_1((M-\mathcal{L}^\vee)\cap T_U)\to H_1(T_U)\simeq \mathbb{Z}$$ is exact. By assumption, there is a nonzero $m\in \mathbb{Z}$ which is in the image of both $I_U(\mathcal{L}^\vee)$ and $I_U(M-\mathcal{L}^\vee)$. By exactness, $m$ is also in the image of $$H_1(\mathcal{S}^\vee\cap T_U)\to H_1(T_U).$$ Let $\gamma\in \mathcal{S}^\vee\cap T_U$ be a closed curve which represents $m$ in $H_1(T_U)$. Apriori, $\gamma$ may not be embedded. We replace $\gamma$ by an embedded curve representing a nonzero class in $H_1(T_U)$ as follows. Suppose $\gamma$ has a self intersection. Then $[\gamma]=[\gamma_1]+[\gamma_2]$, for two curves $\gamma_1,\gamma_2\subset \gamma$ as subsets. Either $[\gamma_1]$ or $[\gamma_2]$ is nonzero. Replacing $\gamma$ by one of these curves decreases the number of self-intersections. Continuing this process, we obtain an embedded $\gamma'\in \mathcal{S}^\vee \cap T_U$ such that $[\gamma']\neq 0\in H_1(T_U)$. This means that $[\gamma']$ represents $ap$ in $H_1(M)\simeq \mathbb{Z}$, for some integer $a\neq 0$. Since $\gamma'$ is embedded and $\mathcal{S}^\vee$ is a sphere with boundary components, $\gamma'$ bounds a disk with boundary components in $\mathcal{S}^\vee$. Because $[\gamma']=ap$, there are at least $|a|p$ boundary components. Since $|a|\geq 1$, this contradicts the Euler characteristic assumption on $\mathcal{S}^\vee$.

\textbf{Case 2.} $I_V(\mathcal{L}^\vee)$ and $I_V(M-\mathcal{L}^\vee)$ are nontrivial. This is analogous to Case 1.

\textbf{Case 3.} $I_U(\mathcal{L}^\vee)$ and $I_V(M-\mathcal{L}^\vee)$ are nontrivial. 

First, we replace $\mathcal{L}^\vee$ and $\mathcal{S}^\vee$ by $\widetilde{\mathcal{L}}\subset S^3$, a $3$-submanifold of $S^3$ with boundary $\widetilde{\mathcal{S}}$, as follows. Because $\mathcal{S}^\vee$ is admissible, it only intersects $\partial M$ in inessential curves, or meridional curves. First, we remove the inessential intersection components by compressing $\mathcal{S}^\vee$ along the disk it bounds (and removing a neighborhood of the disk it bounds from $\mathcal{L}^\vee$.) Second, note that $\partial M \cap \mathcal{L}^\vee$ is a disjoint set of annuli. We replace each such annulus by tubes in $N(T_{p,p+1})$ to remove the meridional intersection components of $\mathcal{S}^\vee\cap \partial M$. After these surgeries, we have a union of spheres $\widetilde{\mathcal{S}}\subset S^3$ bounding a $3$-submanifold $\widetilde{\mathcal{L}}\subset S^3$.

Since the surgeries only take place in a neighborhood of $N(T_{p,p+1})$, we still have that the inclusions $$I_U(\widetilde{\mathcal{L}}):H_1(\widetilde{\mathcal{L}}\cap T_U)\to H_1(T_U)$$ and $$I_V(S^3-\widetilde{\mathcal{L}}):H_1((S^3-\widetilde{\mathcal{L}})\cap T_V)\to H_1(T_V)$$ are nontrivial. Since $\widetilde{\mathcal{S}}\subset S^3$ is a union of spheres, $H_1(\widetilde{\mathcal{L}})=0$. Let $\gamma_1$ be a closed curve in $\widetilde{\mathcal{L}}\cap T_U$ such that $[\gamma_1]\neq 0\in H_1(T_V)$ (which exists since $I_U(\widetilde{\mathcal{L}})$ is nontrivial). Let $\gamma_2$ be a closed curve in $(S^3-\widetilde{\mathcal{L}})\cap T_V$ such that $[\gamma_1]\neq 0\in H_1(T_V)$ (which exists since $I_V(S^3-\widetilde{\mathcal{L}})$ is nontrivial). Note that $\gamma_1$ and $\gamma_2$ have a nonzero linking number. So $[\gamma_1]\neq 0\in H_1(S^3-\gamma_2)$. Since $\widetilde{\mathcal{L}}\subset S^3-\gamma_2$ and $[\gamma_1]=0$ in $H_1(\widetilde{\mathcal{L}})$, we have a contradiction. 

\textbf{Case 4.} $I_V(\mathcal{L}^\vee)$ and $I_U(M-\mathcal{L}^\vee)$ are nontrivial. This is analogous to Case 3.

\subsection{Proof of \cref{general topological lemma}: combinatorics} \label{proof of topological lemma: combinatorics}

In this section, we prove \cref{combinatorial lemma for cyclic splitting}, completing the proof of \cref{general topological lemma}.

Recall that we have $\Delta=\{\Omega_{-\ell},...,\Omega_k\}$, an $f$-twisted essential discrete path of subsurfaces of $\Sigma_{p,p+1}$ that we obtained from a $3$-submanifold of $M$. We also have $\Delta_U=\{\Omega_U^t\}$ and $\Delta_V=\{\Omega_V^t\}$, which are $\phi$-twisted and $\psi$-twisted continuous paths of surfaces on $U$ and $V$, respectively. The $3$-dimensional realizations of $\Delta_U$ and $\Delta_V$ glue to give a $3$-submanifold isotopic to a $3$-dimensional realization of $\Delta$. 

By \cref{eq 4} and additivity of the adjusted Euler characteristic, $$\frac{p(p-1)}{10}\leq|\chi^{p,p+1}(\Omega_k), \chi^{p,p+1}(\Sigma_{p,p+1}-\Omega_k)|\leq \frac{9p(p-1)}{10}.$$ By \cref{chi is additive}, 
$$\sum_{i=1}^p\chi^{U,P_U}(\Omega_k\cap U_i)+\sum_{j=1}^{p+1}\chi^{V,P_V}(\Omega_k\cap V_j)\geq \frac{p(p-1)}{10}.$$ So either 
\begin{equation}\label{eq 6}\sum_{i=1}^p\chi^{U,P_U}(\Omega_k\cap U_i)\geq \frac{p(p-1)}{20},
\end{equation} or 
\begin{equation}\label{eq 7}\sum_{i=1}^{p+1}\chi^{V,P_V}(\Omega_k\cap V_i)\geq \frac{p(p-1)}{20}.
\end{equation} Similarly, either 
\begin{equation}\label{eq 8}\sum_{i=1}^p\chi^{U,P_U}(U-(\Omega_k\cap U_i))\geq \frac{p(p-1)}{20},
\end{equation} or 
\begin{equation}\label{eq 9} \sum_{i=1}^{p+1}\chi^{V,P_V}(V-(\Omega_k\cap V_i))\geq \frac{p(p-1)}{20}.
\end{equation} We will now show that if \cref{eq 6} holds, then $\Delta_U$ does not split with respect to $\phi$. Analogous statements will hold for \cref{eq 7}, \cref{eq 8} and \cref{eq 9} with $\Delta_V$, $U-\Delta_U$ and $V-\Delta_V$, respectively. Together these statements will imply \cref{combinatorial lemma for cyclic splitting}.

Assume \cref{eq 6} holds. By \cref{eq 5}, $$\length(\Delta)\leq \frac{p}{10^4}.$$ By condition (7) in \cref{U and V sequences}, 
\begin{equation}\label{eq 10}\length([\Delta_U])\leq \frac{p}{1000}.
\end{equation}

Recall that $\Delta_U=\{\Omega_U^t\}$. By condition (1) in \cref{U and V sequences}, for each $1\leq i\leq p$, there exists $t\in [0,1]$ such that $\Omega_U^t$ is isotopic to $\Omega_k\cap U_i$. By \cref{eq 6}, there exists $t\in [0,1]$ such that 
\begin{equation}\label{eq 11}\chi^{U,P_U}(\Omega_U^t)\geq \frac{p}{20}.
\end{equation}

\begin{lemma}\label{last cyclic splitting lemma} Let $D$ be a closed disk with $P$ a set of $2n$ points in $\partial D$. Let $g:(D,P)\to (D,P)$ be the homeomorphism that shifts points in $P$ by two in the direction of the orientation. Let $\Delta=\{\Omega^t\}$ be a $g$-twisted continuous path of subsurfaces of $(D,P)$, with $$\length([\Delta])\leq \frac{n}{1000}.$$ If $\Delta$ splits with respect to $g$, then $$\chi^{D,P}(\Omega^t)\leq \frac{n}{20}$$ for all $t\in [0,1]$. 
\end{lemma}

Let us see how \cref{last cyclic splitting lemma} implies \cref{combinatorial lemma for cyclic splitting}. Applying the lemma to $\Delta_U$ proves that if \cref{eq 6} holds, then so does \cref{eq 11}, so $\Delta_U$ does not split with respect to $\phi$. Similarly, we may show that if \cref{eq 8} holds, then $U-\Delta_U$ does not split with respect to $\phi$. Flipping $\Delta_V$ to reverse the orientation, we may show that if \cref{eq 7} holds (resp. \cref{eq 9} holds), then $\Delta_V$ does not split with respect to $\psi$ (resp. $V-\Delta_V$ does not split with respect to $\psi$).

So to complete the proof of \cref{combinatorial lemma for cyclic splitting}, it remains to prove \cref{last cyclic splitting lemma}. In order to do this, we introduce some preliminary combinatorial definitions and lemmas. 

Label the points in $$P=\{x_1,...,x_{2n}\}$$ in order. Label the edges of $\partial D$ (i.e. components of $\partial D-P$) by $$e_1,...,e_{2n},$$ so that $e_i$ is the edge between $x_i$ and $x_{i+1}$, and $e_{2n}$ is the edge between $x_{2n}$ and $x_1$. 

\begin{definition}\label{minimal connected pair} Given a connected subsurface $\Omega\subset (D,P)$ and a pair $(i,j)$ with $i<j\in \{1,...,2n\}$, we say that $(i,j)$ is a connected pair if:

\begin{enumerate}
\item $j-i\geq 3$ and $i+(2n-j)\geq 3$. (In other words, the cyclic distance between $i$ and $j$ is at least $3$.)

\item $\Omega$ intersects both $e_i$ and $e_j$.
\end{enumerate}

We say that $(i,j)$ is a minimal connected pair if either for all $i<k<j$, neither $(i,k)$ nor $(k,j)$ is a connected pair; or for all $k<i$ and $k>j$, neither $(k,i)$ nor $(j,k)$ is a connected pair.

Now, given a (not necessarily connected) subsurface $\Omega\subset (D,P)$ and a pair $(i,j)$ with $i<j\in \{1,...,2n\}$, we say that $(i,j)$ is a minimal connected pair associated to $\Omega$ if $(i,j)$ is a minimal connected pair associated to some connected component of $\Omega$. Let $$M(\Omega)\subset \{(i,j)\in \{1,...,2n\}^2|i<j\}$$ be the set of minimal connected pairs associated to $\Omega$.
\end{definition}

\begin{lemma} \label{sublemma 1} If $\Omega$ and $\Omega'$ are subsurfaces of $(D,P)$ with $d(\Omega,\Omega')\leq 1$, then $$|M(\Omega)-(M(\Omega)\cap M(\Omega'))|\leq 6.$$
\end{lemma}

\begin{proof} First, we consider the case wherein $\Omega$ and $\Omega'$ are related by a surgery along a curve or arc $\delta$. It suffices to consider the case where $\delta\subset \Omega$. (When $\delta\subset D-\Omega$, $M(\Omega')\supset M(\Omega)$, so the lemma follows trivially.) In this situation, we may also assume that $\Omega$ is connected. (If $\Omega$ is disconnected, the lemma follows by considering each connected component separately.)

Let $E(\Omega)$ be the set of edges $\{e_1,...,e_{2n}\}$ that $\Omega$ intersects nontrivially. When we apply a surgery to $\Omega$, either $\Omega'$ remains connected, or it splits into two components. In the first case, it has the same set of edges it intersects with as $\Omega$. Therefore $M(\Omega)=M(\Omega')$. In the second case, let $\Omega_1'$ and $\Omega_2'$ be the two connected components of $\Omega'$. Let $E(\Omega_1')$ and $E(\Omega_2')$ be the set of edges that $\Omega_1'$ and $\Omega_2'$ intersect, respectively. Then $E(\Omega_1')\cup E(\Omega_2')=E(\Omega)$. Moreover, since $\Omega_1'$ and $\Omega_2'$ are disjoint, the edges in $E(\Omega_1')$ are a sequence of edges consecutive in $E(\Omega)$ with respect to the cyclic order on them. Similarly, the edges in $E(\Omega_2')$ are a sequence of edges consecutive in $E(\Omega)$ with respect to the cyclic order on them. Now, $M(\Omega)$ consists of certain pairs of edges in $E(\Omega)$ that are at least cyclic distance $3$ apart. If both edges in a pair are in $E(\Omega_1')$ (resp. $E(\Omega_2')$), the pair will still be a minimal connected pair associated to $\Omega'$. So there are at most six pairs in $M(\Omega)$ that are not in $M(\Omega')$.

Finally, we consider the case wherein $\Omega$ and $\Omega'$ are related by $\partial$-surgery of type 1 or 2. Again, we may assume that $\Omega$ is connected, otherwise treating each connected component separately. In this case, $E(\Omega)$ and $E(\Omega')$ differ by at most one edge. As before, $M(\Omega)$ consists of certain pairs of edges in $E(\Omega)$ that are at least cyclic distance $3$ apart, and at most two pairs can fail to be in $M(\Omega')$. Hence the lemma follows in this case also.
\end{proof}

\begin{lemma} If $\Omega\subset (D,P)$ is a subsurface, then \label{sublemma 2} $|M(\Omega)|\geq |\chi^{D,P}(\Omega)|$.
\end{lemma}

\begin{proof} It suffices to prove the lemma for essential subsurfaces, as applying disk additions or eliminations does not change the set $M(\Omega)$. By definition and additivity, it suffices to prove the lemma in the case where $\Omega$ is connected. Note that $$|\chi^{D,P}(\Omega)|=\frac{|P\cap \Omega|+ |\bound{\Omega}\cap \partial D|}{4}-1.$$ Let $E(\Omega)$ again be the set of edges that $\Omega$ intersects nontrivially. Because $\Omega$ is connected and essential, each edge contains at most two points in $(P\cap \Omega) \cup (\delta(\Omega)\cap \partial D)$. So $$|E(\Omega)|\geq \frac{|P\cap \Omega|+ |\bound{\Omega}\cap \partial D|}{2}.$$ To prove the lemma, it suffices to show that $$|M(\Omega)|\geq \frac{|E(\Omega)|}{2}-1.$$ 

Since $\Omega$ is essential and $|\chi^{D,P}(\Omega)|>0$, $|E(\Omega)|\geq 3$. We split into cases based on the value of $|E(\Omega)|$. For $|E(\Omega)|\geq 5$, for each edge in $E(\Omega)$, there is another edge in $E(\Omega)$ at least cyclic distance $3$ away from the original edge. Hence, in this case, each edge in $E(\Omega)$ is part of a minimal connected pair. Therefore, $|M(\Omega)|\geq |E(\Omega)|/2$. For $|E(\Omega)|=4$, at least one pair of edges must have cyclic distance $3$. So $|M(\Omega)|\geq 1$, as desired. Finally, for $|E(\Omega)|=3$: if there is no minimal connected pair associated to $\Omega$, then the three edges are consecutive. This contradicts $\Omega$ being essential.
\end{proof}

\begin{proof}[Proof of \cref{last cyclic splitting lemma}] Let $\Delta=\{\Omega^t\}$ be a $g$-twisted continuous path of subsurfaces of $(D,P)$ with a length bound. Assume the lemma is false.. Then there exists some $t\in [0,1]$ for which $$\chi^{D,P}(\Omega^t)\geq \frac{n}{20}.$$ Shifting $t$ as necessary (since $D$ is $g$-twisted), we may assume that $t=0$. By the length bound, $$d([\Omega^0],[g(\Omega^0)]])\leq \frac{n}{10^4}.$$ Now, $\Delta=\coprod \Delta_n$ for $\Delta_n=\{\Omega_n^t\}$ satisfying $g(\Omega_{n-1}^0)=\Omega_n^1$. Let $$\Delta_{\even}=\coprod_{n\in 2\mathbb{Z}} \Delta_n$$ and $$\Delta_{\odd}=\coprod_{n\in 2\mathbb{Z}+1}\Delta_n$$ be disjoint unions. Note that $\Delta=\Delta_{\even}\amalg \Delta_{\odd}$. Also, $g(\Omega_{\even}^0)=\Omega_{\odd}^1$ and $g(\Omega_{\odd}^0)=\Omega_{\even}^1$.

\begin{lemma}\label{distance between surface and rotated surface} Let $\Omega\subset (D,P)$ be an essential subsurface. Then $$d(\Omega',g(\Omega))\geq \frac{\chi^{D,P}(\Omega)}{6}$$ for any subsurface $\Omega'\subset (D,P)$ disjoint from $\Omega$. 
\end{lemma}

\begin{proof} Suppose $$d(\Omega',g(\Omega))< \frac{\chi^{D,P}(\Omega)}{6}.$$ By \cref{sublemma 1} and \cref{sublemma 2}, $$M(\Omega')\cap M(g(\Omega))\neq \emptyset$$ (see \cref{minimal connected pair}). Let $(i,j)$ be a minimal connected pair associated to both $\Omega'$ and $g(\Omega)$. Then $\Omega'$ intersects $e_i$ and $e_j$, while $\Omega$ intersects $e_{i+2(\modulo 2n)}$ and $e_{j+2(\modulo 2n)}$. Since the cyclic distance between $i$ and $j$ is at least $3$, this means $\Omega$ and $\Omega'$ are not disjoint, which is a contradiction. The case of $V$ is analogous.
\end{proof}

We return to the proof of \cref{last cyclic splitting lemma}. Since $$\chi^{D,P}(\Omega^0)\geq \frac{n}{20},$$ either $$\chi^{D,P}(\Omega^0_{\even})\geq \frac{n}{40}$$ or $$\chi^{D,P}(\Omega^0_{\odd})\geq \frac{n}{40}.$$ Without loss of generality, we assume the former. Note $\Omega^0_{\even}$ and $\Omega^0_{\odd}$ are disjoint, so $\widehat{[\Omega^0_{\even}]}$ and $\widehat{[\Omega^0_{\odd}]}$ are also disjoint. Since $\Omega^1_{\odd}=g(\Omega^0_{\even})$, $$\widehat{[\Omega^1_{\odd}]}=g(\widehat{[\Omega^0_{\even}]}).$$ Since $$\length([\Delta_{\odd}])\leq \length(\Delta)\leq \frac{n}{1000}, $$ $$d([\Omega^0_{\odd}],[\Omega^1_{\odd}])\leq \frac{n}{1000}.$$ By \cref{distance between null eliminated surfaces 2}, $$d(\widehat{[\Omega^0_{\odd}]},\widehat{[\Omega^1_{\odd}]})\leq \frac{3n}{1000}.$$ This contradicts \cref{distance between surface and rotated surface}.
\end{proof}

\section{Proof of the main theorem} \label{Proofs of the main theorems}

In this section, we prove \cref{distortion of connected sum is large}. It is a consequence of \cref{topological lemma for connected sum}, along with an adaptation of the double bubble argument from \cite{Par11}.

\subsection{Double bubble argument} 

\begin{definition} A double bubble $Z$ is an embedded $3$-complex in $\mathbb{R}^3$ constructed as follows. The $1$-skeleton, $Z^1$, is an embedded $S^1$ in $\mathbb{R}^3$. The $2$-skeleton, $Z^2$, consists of three embedded disks with disjoint interiors, whose boundaries are $Z^1$. We denote the three $2$-cells by $e_{Z,1}^2$, $e_{Z,2}^2$ and $e_{Z,3}^2$. There are two $3$-cells, denoted $e_{Z,1}^3$ and $e_{Z,2}^3$, both embedded balls in $\mathbb{R}^3$, with $$\partial e_{Z,1}^3=e_{Z,1}^2\cup e_{Z,2}^2$$ and $$\partial e_{Z,2}^3= e_{Z,2}^2\cup e_{Z,3}^2.$$  
\end{definition}

\begin{lemma} \label{double bubble lemma} Let $\beta$ be an embedding of $T_{p,p+1}\# K$ in $\mathbb{R}^3$. Let $\Sigma$ be a minimal genus Seifert surface with $\partial \Sigma=\beta$. Let $Z$ be a double bubble intersecting $\beta$ and $\Sigma$ transversely, such that $$|Z^2\cap \beta|<\frac{p}{2\cdot 10^4}.$$ If $$|\chi_{\Sigma}^{p,p+1}(e_{Z,1}^3\cap \Sigma)|+|\chi_{\Sigma}^{p,p+1}(e_{Z,2}^3\cap \Sigma)|\leq \frac{p(p-1)}{10},$$ then $$|\chi_{\Sigma}^{p,p+1}((e_{Z,1}^3 \cup e_{Z,2}^3)\cap \Sigma)|\leq \frac{p(p-1)}{10}.$$ Analogously, if $$|\chi_{\Sigma}^{p,p+1}((e_{Z,1}^3 \cup e_{Z,2}^3)\cap \Sigma)|\geq \frac{9p(p-1)}{10},$$ then $$|\chi_{\Sigma}^{p,p+1}(e_{Z,1}^3\cap \Sigma)|+|\chi_{\Sigma}^{p,p+1}(e_{Z,2}^3\cap \Sigma)|\geq \frac{9p(p-1)}{10}.$$ 
\end{lemma}

This lemma is similar to Lemma 2.7 in \cite{Par11}.

\begin{proof} Consider the $1$-parameter family of spheres $S_t$, starting from $S_0=e_{Z,1}^2\cup e_{Z,3}^2$ and doing surgery along $e_{Z,2}^2$. By a suitable perturbation we may assume that $S_t$ is transverse to $\Sigma$ except for a finite number of values of $t$ at which point $S_t\cap \Sigma$ undergoes an elementary move. We can also arrange that $$|S_t\cap \beta|<\frac{p}{10^4}$$ for all $t$. 

Now, suppose $$|\chi_{\Sigma}^{p,p+1}(e_{Z,1}^3\cap \Sigma)|+|\chi_{\Sigma}^{p,p+1}(e_{Z,2}^3\cap \Sigma)|\leq \frac{p(p-1)}{10}$$ and $$|\chi_{\Sigma}^{p,p+1}((e_{Z,1}^3 \cup e_{Z,2}^3)\cap \Sigma)|> \frac{p(p-1)}{10}.$$ By \cref{chi only depends on equivalence class 2} and \cref{elementary move on chi 2}, for some $t\in [0,1)$, $$\frac{p(p-1)}{10}\leq |\chi_{\Sigma}^{p,p+1}(\interior(S_t)\cap \Sigma)|\leq \frac{p(p-1)}{10}+2.$$ This contradicts \cref{topological lemma for connected sum}. Similarly, suppose $$|\chi_{\Sigma}^{p,p+1}((e_{Z,1}^3 \cup e_{Z,2}^3)\cap \Sigma)|\geq \frac{9p(p-1)}{10}$$ and $$|\chi_{\Sigma}^{p,p+1}(e_{Z,1}^3\cap \Sigma)|+|\chi_{\Sigma}^{p,p+1}(e_{Z,2}^3\cap \Sigma)|< \frac{9p(p-1)}{10}.$$ Then for some $t\in [0,1)$, $$\frac{9p(p-1)}{10}-2\leq |\chi_{\Sigma}^{p,p+1}(\interior(S_t)\cap \Sigma)|\leq \frac{9p(p-1)}{10},$$ which is a contradiction.
\end{proof}

\subsection{Conformal length and proof of \cref{distortion of connected sum is large}} 

We begin by proving a comparison between conformal length, diametrical distortion and distortion as described in \cref{topology underlying distortion}.

\begin{lemma}\label{distortion comparison} $$\conlength(\beta),\ddist(\beta)\lesssim \dist(\beta)\lesssim \conlength(\beta)\cdot \ddist(\beta).$$
\end{lemma}

\begin{proof} By definition, $$\ddist(\beta)\lesssim \dist(\beta).$$ The inequality $$\conlength(\beta)\lesssim \dist(\beta)$$ is proved in \cite[Lemma 4.1]{GG12}. It remains to show that 

\begin{equation}\label{eq 2.0}\dist(\beta)\lesssim \conlength(\beta)\cdot \ddist(\beta).
\end{equation} Let $x,y\in \beta$ such that $\displaystyle\frac{d_{\beta}(x,y)}{|x-y|}$ is maximized. By definition, $\diam_\beta(x,y)\leq d_{\beta}(x,y)$. Now, $$\conlength(\beta)\gtrsim \frac{d_{\beta}(x,y)}{\diam_\beta(x,y)},$$ by considering a ball of radius around $\diam_\beta(x,y)$ around $x$. As a result, we have $$\conlength(\beta)\cdot \frac{\diam_\beta(x,y)}{|x-y|}\gtrsim \frac{d_{\beta}(x,y)}{|x-y|}.$$ \cref{eq 2.0} follows.
\end{proof}

Our proof of \cref{distortion of connected sum is large} goes through conformal length, like the arguments in \cite{Par11} and \cite{GG12}.

\begin{theorem}\label{conformal length of connected sum is large} Let $\beta\subset \mathbb{R}^3$ be an embedding of $T_{p,p+1}\# K$. Then $$\conlength(\beta)\gtrsim p,$$ with constant independent of $p$ and $K$.
\end{theorem}

\begin{proof} Let $\beta\subset \mathbb{R}^3$ be an embedding of $T_{p,p+1}\# K$, and assume that $$\conlength(\beta)\leq \frac{p}{5\cdot 10^5}.$$ Let $\Sigma$ be any minimal genus Seifert suface with $\partial \Sigma=\beta$. 

We now consider (closed) boxes $B$ of dimensions $r, 2^{1/3}r, 2^{2/3}r$, for $r>0$. We call a box with such dimensions a box of scale $r$. Let $\mathcal{S}$ be the set of all boxes $B$ of scale $r$ such that
$$|\chi_{\Sigma}^{p,p+1}(B\cap \Sigma)|\geq \frac{9p(p-1)}{10}.
$$ Let $Q$ be a box of scale $r$ in the set $\mathcal(S)$ that is approximately the smallest, meaning that $\mathcal{S}$ contains no boxes of scale $(2/3)r$. We will obtain a contradiction to the assumption that $Q$ is an approximately smallest element of $\mathcal{S}$.

Scale $\beta$ and $Q$ so that $Q$ has dimension $1$. Translate $Q$ so that it is centered at the origin $(0,0,0)$. By our assumption on the conformal length of $\beta$, 
\begin{equation}\label{eq 13} \length(\beta\cap Q)\leq \frac{p}{5\cdot 10^5}.
\end{equation} Apriori, $\beta$ and $\Sigma$ may not be smooth, but we may replace them with arbitrarily close smooth copies so that \cref{eq 13} holds, and $Q$ is still an approximately smallest element of $\mathcal{S}$.  

For $r\in [1,6/5]$, let $Q_r$ be the box with center at the origin and dimension $r, 2^{1/3}r, 2^{2/3}r$. Then \begin{align*}\int_1^{6/5}|\partial Q_r\cap \beta|dr &\leq \length(Q_{6/5}\cap \beta)\\&\leq \length(B((0,0,0),3)\cap \beta)\\&\leq 3\conlength(\beta)\\&\leq \frac{3p}{5\cdot 10^5}.
\end{align*} So for some $Q'$ among the $Q_r$s that intersects $\beta$ and $\Sigma$ transversely, $$|\partial Q'\cap \beta|\leq\frac{3p}{10^5}.$$

Next, we find a plane intersecting the long dimension of $Q'$ whose intersection with $\beta$ is bounded. Without loss of generality, assume $z$ is the long dimension of $Q'$. let $P_s$ be the intersection of the plane $z=s$ with the $Q'$. Then
\begin{align*} \int_{-1/10}^{1/10}|P_s\cap \beta|ds&\leq \length(Q_{6/5}\cap \beta)\\&\leq \frac{3p}{5\cdot 10^5}
\end{align*} 
as before. Therefore, for some $s\in [-1/10,1/10]$, $P_s$ intersects $\beta$ and $\Sigma$ transversely and $$|P_s\cap \beta|\leq \frac{3p}{10^5}.$$

The plane $P_s$ divides $Q'$ into two boxes, which we label $Q'_1$ and $Q'_2$. Each box is contained in a box with dimension $2/3, (2/3)2^{1/3}, (2/3)2^{2/3}$. By  \cref{double bubble lemma}, $$|\chi_{\Sigma}^{p,p+1}(Q'_i\cap \Sigma)|\geq \frac{p(p-1)}{10}$$ for some $i\in \{1,2\}$. This contradicts the assumption that $Q$ is the smallest such box.
\end{proof}

\appendix

\section{Extrinsic systole of Seifert surfaces}

In \cite{Vas25}, we introduced the extrinsic systole, a new extrinsic invariant in systolic geometry. In this appendix, we explain how the proof of \cref{distortion of connected sum is large} can be modified to obtain an upper bound for the extrinsic systole of Seifert surfaces associated to $T_{p,p+1}\# K$. The latter bound is also a consequence of \cref{topological lemma for connected sum}.

The systole of a Riemannian manifold is the length of its shortest non-contractible closed curve. The extrinsic systole generalizes the systole to manifolds embedded in an ambient Riemannian space. For surfaces in $\mathbb{R}^3$, it has the following definition. 

\begin{definition} Let $\Sigma\subset \mathbb{R}^3$ be an embedded surface, with $\pi_1(\Sigma)$ nontrivial. The extrinsic systole of $\Sigma$, denoted $\sys_{\mathbb{R}^3}(\Sigma)$, is the diameter of the smallest ball in $\mathbb{R}^3$ that contains a non-contractible closed curve on $\Sigma$.
\end{definition}

As a consequence of \cref{topological lemma for connected sum}, we have the following upper bound for the extrinsic systole of minimal genus Seifert surfaces of $T_{p,p+1}\# K$:

\begin{theorem}\label{Seifert surface is thin connected sum}  Let $\beta\subset \mathbb{R}^3$ be an embedding of $T_{p, p+1}\# K$. Let $\Sigma\subset \mathbb{R}^3$ be any $C^1$-embedded surface of minimal genus with $\partial \Sigma=\beta$. Then $$\textstyle\sys_{\mathbb{R}^3}(\Sigma)\lesssim \displaystyle\frac{\length(\beta)}{p}.$$  
\end{theorem}

Like \cref{distortion of connected sum is large}, \cref{Seifert surface is thin connected sum} also follows from \cref{topological lemma for connected sum} and a double bubble argument. The connection between distortion of knots and extrinsic systole of surfaces is broader; in \cite[Theorem A.1]{Vas25}, we modified Pardon's argument in \cite{Par11} to prove an extrinsic systolic inequality for embedded tori in $\mathbb{R}^3$. Roughly speaking, the distortion lower bounds and the extrinsic systole upper bounds are different quantitative consequences of the same underlying topological statements.

\subsection{Proof of \cref{Seifert surface is thin connected sum}} Perturbing $\beta$ and $\Sigma$ slightly, we assume that they are both smooth. Scaling appropriately, we assume $\length(\beta)=1$. Suppose that $$\textstyle\sys_{\mathbb{R}^3}(\Sigma)> \displaystyle \frac{3^{1/2}\cdot 2\cdot 10^4}{p}.$$ We will show a contradiction. 

Let $X$ be the $2$-skeleton of the cubical lattice in $\mathbb{R}^3$ with side length $2\cdot 10^4/p$. First, we translate $\beta$ so that the intersection with $X$ is bounded above, and $X$ intersects $\beta$ and $\Sigma$ transversely. To do this, note that $X$ is the union of planes perpendicular to the $x$ direction, planes perpendicular to the $y$ direction, and planes perpendicular to the $z$ direction, which we denote by $X_x$, $X_y$ and $X_z$, respectively. 

As $\beta$ is translated in the $y$ or $z$ direction, $X_x\cap (\beta+s)$ does not change. Hence, 
\begin{align*}\int_{s\in [0,2\cdot 10^4/p]^3}|X_x\cap (\beta+s)|ds&\leq \length(\beta)\\& =\left(\frac{2\cdot 10^4}{p}\right)^2.
\end{align*} Similar inequalities hold for $X_y$ and $X_z$. Adding, we obtain $$\int_{s\in [0,2\cdot 10^4/p]^3}|X\cap (\beta+s)|ds\leq 3\left(\frac{2\cdot 10^4}{p}\right)^2.$$ So suitably translating $\beta$, we may assume that $\beta$ and $\Sigma$ intersect $X$ transversely, and $$|X\cap \beta|\leq \frac{3p}{2\cdot 10^4}.$$

By our assumption that $$\textstyle\sys_{\mathbb{R}^3}(\Sigma)\geq \displaystyle \frac{3^{1/2}\cdot 2\cdot 10^4}{p},$$ no cube in the complement of $X$ contains any non-contractible curve on $\Sigma$.

\begin{lemma}\label{complement of lattice does not intersect surface much} We have, $$|\chi_{\Sigma}^{p,p+1}(X^c\cap \Sigma)|\leq \frac{3p}{2\cdot 10^4}.$$ 
\end{lemma}

\begin{proof} It suffices to prove the lemma replacing $X^c$ by $N(X)^c$, the complement of an arbitrarily small neighborhood of $X$. Note that $\partial N(X)^c$ is a union of spheres. By choosing the neighborhood to be sufficiently small, we can arrange that $$|\partial N(X)^c\cap \beta|\leq \frac{3p}{2\cdot 10^4}.$$ 

Let $\Omega$ be the subsurface $N(X)^c\cap \Sigma$ of $\Sigma$. Then $$|\bound{\Omega}\cap \partial \Sigma|\leq \frac{3p}{2\cdot 10^4}.$$ Since $\Omega$ does not contain any non-contractible curve of $\Sigma$, $\chi(\Omega^{\ess})\geq 0$. Therefore, $$|\chi_\Sigma(\Omega)|\leq \frac{3p}{2\cdot 10^4}.$$  Note that $\Sigma=\Sigma_{p,p+1}\# \Sigma_K$, where the connected sum is along an arc $\delta$ on $\Sigma_K$. Put $\Omega$ in minimal position with $\delta$. By \cref{chi is additive}, $$\chi_{\Sigma}^{p,p+1}(\Omega)+\chi_{\Sigma}^{K}(\Omega)=\chi_{\Sigma}(\Omega)\geq \frac{3p}{2\cdot 10^4}.$$ Since $\chi$ is non-positive, $$\chi_{\Sigma}^{p,p+1}(\Omega)\geq \frac{3p}{2\cdot 10^4}.$$ The lemma follows.
\end{proof}

Let $Q$ be a cube (with $\partial Q\subset X$) of side length in $[2- 2\cdot 10^4/p,2+2\cdot 10^4/p]$, containing $\beta$ in its interior. Such a cube exists since $\length(\beta)=1$. Since $\Sigma$ is a Seifert surface of minimal genus, $\partial Q\cap \Sigma$ is simply the union of some inessential closed curves on $\Sigma$. Thus 
\begin{equation}\label{eq 12}|\chi_{\Sigma}^{p,p+1}(Q\cap \Sigma)|=p(p-1)-\frac{1}{2}.
\end{equation} Consider $X\cap Q$, the $2$-skeleton of the cubical lattice restricted to $Q$. The complex $X\cap Q$ is an iterated double bubble; \cref{complement of lattice does not intersect surface much} along with repeated applications of \cref{double bubble lemma} gives a contradiction with \cref{eq 12}.

\bibliographystyle{alpha} 

\bibliography{bibliography}

\end{document}